\documentclass[12pt,a4paper]{amsart}
\DeclareMathSizes{12}{12}{7}{6}
\usepackage[final=true]{hyperref}
\usepackage{amsmath,amsrefs,amsthm,amssymb}
\usepackage{tikz-cd}
\usepackage{enumitem}
\usepackage[capitalize]{cleveref}
\usepackage{newtxtext}
\usepackage{mathtools}
\usepackage{tikz}
\usepackage{MnSymbol}
\usepackage[normalem]{ulem}
\usepackage{cancel}

\crefname{figure}{Figure}{Figures} 
\Crefname{figure}{Figure}{Figures}

\newcommand{\B}{\mathcal{B}}
\newcommand{\C}{\mathcal{C}}
\DeclareMathOperator{\D}{\mathcal{D}}

\renewcommand{\S}{\mathcal{S}}

\newcommand{\U}{\mathcal{U}}
\newcommand{\dU}{{}^{\perp_d}{\U}}
\newcommand{\Ud}{\U^{\perp_d}}
\newcommand{\V}{\mathcal{V}}
\newcommand{\Z}{\mathbb{Z}}

\newcommand{\Ext}{\operatorname{Ext}}
\newcommand{\Hom}{\operatorname{Hom}}
\newcommand{\add}{\operatorname{add}}
\newcommand{\proj}{\operatorname{proj}}

\tikzset{
petal/.style={blue, very thick},
laminate/.style={red},
arc/.style={black, very thick}
}

\newtheorem{thmu}{Theorem}
\newtheorem{theorem}{Theorem}[section]
\newtheorem{lemma}[theorem]{Lemma}
\newtheorem{corollary}[theorem]{Corollary}
\newtheorem{proposition}[theorem]{Proposition}

\theoremstyle{definition}
\newtheorem*{acknowledgements}{Acknowledgements}
\newtheorem{definition}[theorem]{Definition}
\newtheorem{remark}[theorem]{Remark}
\newtheorem{example}[theorem]{Example}

	\newcommand{{\sib}}[1]{\textcolor{blue}{#1}}

\title{The role of gentle algebras in higher homological algebra}
\author[Haugland]{Johanne Haugland}
        \address{Department of Mathematical Sciences\\ 
        NTNU\\ 
        NO-7491 Trondheim\\ 
        Norway}
        \email{johanne.haugland@ntnu.no}
        
\author[Jacobsen]{Karin M. Jacobsen}
        \address{Department of Mathematics\\ Aarhus Universitet\\ Ny Munkegade 118\\ DK-8000 Aarhus C\\ Denmark}
        \email{karin.jacobsen@ntnu.no}
        
\author[Schroll]{Sibylle Schroll}
        \address{Department of Mathematics\\ University of Cologne\\ 50931 Cologne \\Germany \\ and  Department of Mathematical Sciences\\ 
        NTNU\\ 
        NO-7491 Trondheim\\ 
        Norway}
        \email{schroll@math.uni-koeln.de}

\begin{document}


\maketitle

\begin{abstract}
    We investigate the role of gentle algebras in higher homological algebra. In the first part of the paper, we show that if the module category of a gentle algebra $\Lambda$ contains a $d$-cluster tilting subcategory for some $d \geq 2$, then $\Lambda$ is a radical square zero Nakayama algebra. This gives a complete classification of weakly $d$-representation finite gentle algebras. In the second part, we use a geometric model of the derived category to prove a similar result in the triangulated setup. More precisely, we show that if $\D^b(\Lambda)$ contains a $d$-cluster tilting subcategory that is closed under $[d]$, then $\Lambda$ is derived equivalent to an algebra of Dynkin type $A$. Furthermore, our approach gives a geometric characterization of all $d$-cluster tilting subcategories of $\D^b(\Lambda)$ that are closed under $[d]$.
\end{abstract}

\section{Introduction}

The research field of higher homological algebra was initiated by Iyama \cites{Iyama2007,Iyama2007_2}. It concerns the study of $d$-abelian and $(d+2)$-angulated categories, as well as further generalizations \cites{Jasso,GKO,HLN}. Distinguished sequences consisting of $d+2$ objects, for a fixed positive integer $d$, play a fundamental role in these structures. In the case $d=1$, one recovers the short exact sequences and distinguished triangles of abelian and triangulated categories, and the theory corresponds to classical homological algebra.

Iyama's work and the axiomatizations of associated categorical structures inspired extensive research activity, and many ideas from classical homological algebra have been shown to have an analogue in the higher setting \cites{J,JJ,JJ20,M,JK,H,R,HJV}. As connections between higher homological algebra and other branches of mathematics have been developed, the importance of the research field has become increasingly evident. Higher homological algebra is intimately related to higher Auslander--Reiten (AR) theory and representation theory of finite dimensional algebras \cites{Iyama2007_2,Iyama2008,Kvamme-Jasso}. It has connections to commutative algebra, commutative and non-commutative algebraic geometry, combinatorics and conformal field theory \cites{AIR,EP,HIMO,IW,OT,W,J18}. The research field has recently seen interesting applications in homological mirror symmetry, through which it relates to symplectic geometry and Fukaya categories \cite{DJL}.

The notion of $d$-cluster tilting subcategories plays a crucial role in higher homological algebra. A $d$-cluster tilting subcategory of an abelian category is $d$-abelian \cite{Jasso}*{Theorem 3.16}, and every $d$-abelian category has been shown to arise in this way \cites{Kvamme,ENI}. Similarly, a $d$-cluster tilting subcategory of a triangulated category carries a $(d+2)$-angulated structure given that it is closed under $d$-suspension \cite{GKO}*{Theorem 1}. We investigate the role of gentle algebras in higher homological algebra by studying the $d$-cluster tilting subcategories both of their module and their derived categories.  
While the questions we answer in this paper are of a higher homological nature, geometric models play a crucial role in our proofs. In particular, it seems difficult to prove our main result without applying the geometric insights offered in \cite{OPS}. Thus, in addition to the new understanding our results provide on the role of gentle algebras in higher homological algebra, this also demonstrates the power of geometric models.

Recall that a finite dimensional algebra is called weakly $d$-representation finite if it has a module that generates a $d$-cluster tilting subcategory. The study of such algebras has played an important role in the development of higher homological algebra as we know it today. From the viewpoint of higher AR-theory, the class of weakly $d$-representation finite algebras can be thought of as a higher analogue of algebras of finite representation type. In particular, the definition coincides with the classical notion of a representation finite algebra in the case $d=1$.

Just as the classification of (hereditary) algebras of finite representation type has been one of the fundamental questions in classical representation theory, the classification of weakly $d$-representation finite algebras is an important question in higher representation theory. In general, this is a difficult problem, but significant progress has been made for particularly nice classes of algebras. Darpö and Iyama characterize weakly $d$-representation finite cyclic Nakayama algebras with homogeneous relations in \cite{DI}*{Theorem 5.1}. The acyclic case was first studied by Jasso \cite{Jasso}*{Proposition 6.2}, and Vaso gives a complete classification in \cite{Vaso}*{Theorem 2}. In the same paper, Vaso also characterizes all $d$-representation finite $d$-hereditary Nakayama algebras \cite{Vaso}*{Theorem 3}. A classification of iterated tilted $d$-representation finite $d$-hereditary algebras in the case $d=2$ is given by Iyama and Oppermann \cite{IO13}*{Theorem 3.12}. Very recently, similar classification results have been obtained in the context of radical square zero algebras \cite{Vaso21}, monomial algebras \cite{ST}, and symmetric algebras \cites{DK}.

A natural question to ask is whether the classification results mentioned above can be extended to more general classes of algebras. Gentle algebras constitute a large class of algebras which naturally extends many of the known examples where a classification has been obtained. In this paper we give a complete classification of weakly $d$-representation finite gentle algebras, as well as $d$-representation finite $d$-hereditary gentle algebras, see \cref{cor:weakly d-RF gentle} and \cref{cor: d-RF gentle}. The main step towards these results is the theorem below, where we show that only very few gentle algebras are weakly $d$-representation finite. More precisely, we prove the following.

\begin{thmu}[see \cref{thm: nakayama}] \label{thm: 1}
Let $\Lambda$ be a gentle algebra. If $\bmod \Lambda$ contains a $d$-cluster tilting subcategory for some $d \geq 2$, then $\Lambda$ is a radical square zero Nakayama algebra.
\end{thmu}

While the existence of $d$-cluster tilting subcategories of module categories is well-studied for certain classes of algebras, less is known in the triangulated setup. The main aim of this paper is to increase this understanding in the case of derived categories associated to gentle algebras. The class of gentle algebras is special in that not only are these algebras of tame representation type, but they are also derived tame. The indecomposable objects in the bounded derived category of a gentle algebra are classified in \cite{BM}, and a basis of the morphism space between indecomposable objects is described in \cite{ALP}. In \cite{OPS} a geometric model for the derived category of a gentle algebra is given, see also \cites{HKK,LP}. 

Using the geometric model, we characterize $d$-cluster tilting subcategories of the derived category of a gentle algebra that are closed under $d$-suspension. Recall that these subcategories give examples of $(d+2)$-angulated categories. The most important step towards the classification is the following theorem.

\begin{thmu}[see \cref{thm:main result}]
Let $\Lambda$ be a gentle algebra. If $\D^b(\Lambda)$ contains a $d$-cluster tilting subcategory that is closed under $[d]$ for some $d\geq 2$, then $\Lambda$ is derived equivalent to an algebra of Dynkin type $A$.
\end{thmu}

A crucial tool in the proof of this result is \cref{prop: tau^-1 rigid}, where we observe that, for any finite dimensional algebra, an indecomposable perfect object contained in a $d$-cluster tilting subcategory that is closed under $d$-suspension has no non-zero morphisms from its AR-translate to itself. In particular, the middle term in the AR-triangle ending in such an object is indecomposable. This excludes a large class of objects. 

Knowing that the only possible examples arise in the type $A$ case, we classify the $d$-cluster tilting subcategories that are closed under $d$-suspension when our algebra is derived equivalent to an algebra of Dynkin type $A$. This classification is given in \cref{prop:type A}. In particular, this yields a geometric interpretation of the examples coming from $d$-representation finite $d$-hereditary gentle algebras. Combining \cref{thm:main result} and \cref{prop:type A}, we see that all $d$-cluster tilting subcategories of the derived category of a gentle algebra that are closed under $d$-suspension arise in this way.

Altogether, our work reveals a lack of $d$-cluster tilting subcategories arising from gentle algebras, both in the module category and the derived category. Through our characterization results, we see that the examples amount to those one can already obtain by studying Nakayama algebras. This suggests that the role of gentle algebras in higher homological algebra is limited, which is surprising due to the otherwise rich theory of this class of algebras. Our results also show that the situation in the derived setup is even more restrictive than in the module category. In particular, the derived category of a cyclic Nakayama algebra never contains $d$-cluster tilting subcategories closed under $d$-suspension, even though the associated module category may contain $d$-cluster tilting subcategories as described in \cite{DI}.

The paper is structured as follows. In \cref{section:background} we give an overview of some necessary background and preliminaries. This includes the definition of $d$-cluster tilting subcategories and notions related to $d$-representation finiteness, as well as an introduction to the geometric model for the derived category of a gentle algebra. In \cref{section: module category of gentle algebra} we present our results related to the module category of a gentle algebra, before we discuss the derived case in \cref{section: derived category of gentle algebra}.

\subsection{Conventions and notation}
Throughout this paper, let $d$ denote a positive integer. We will typically assume $d \geq 2$. All algebras considered are connected and finite dimensional over an algebraically closed field $K$. The field is assumed to be algebraically closed to be consistent with \cite{ALP}, but as noted in that paper, this condition could be omitted.

Given a quiver $Q$, we denote its set of vertices by $Q_0$ and its set of arrows by $Q_1$. For an arrow $\alpha$ in $Q_1$, we write $s(\alpha)$ for the start vertex of $\alpha$ and $t(\alpha)$ for the end vertex of $\alpha$. Given an arrow $\beta$ with $t(\alpha) = s(\beta)$, we write $\alpha \beta$ for the non-zero product in the path algebra $KQ$. 

We denote the category of finitely generated right modules over an algebra $\Lambda$ by $\bmod \Lambda$. The subcategory of projectives in $\bmod \Lambda$ is denoted $\proj \Lambda$. We use the notation $\D^b (\Lambda)$ for the bounded derived category of $\bmod \Lambda$. The AR-translation (where it exists) is denoted by $\tau$ and the suspension functor in $\D^b(\Lambda)$ by $[1]$.

All subcategories are assumed to be full. Given a set of objects $\S$ in an additive category $\C$, we use the notation $\add \S$ for the subcategory of $\C$ consisting of direct summands of finite direct sums of objects in $\S$.

\section{Background and preliminaries}\label{section:background}

\subsection{\texorpdfstring{$d$}{d}-cluster tilting subcategories} \label{subsec: d-CT}
The notions of $d$-abelian and $(d+2)$-angulated categories were introduced in \cites{GKO, Jasso} to axiomatize properties of $d$-cluster tilting subcategories of abelian and triangulated categories. The definition of such subcategories plays a crucial role in this paper.

Before giving the definition, let us recall what it means for a subcategory $\U$ of some category $\C$ to be functorially finite. Given an object $X$ in $\C$, a morphism $f \colon U \rightarrow X$ with $U$ in $\U$ is a \emph{right $\U$-approximation} of $X$ if any morphism $U' \rightarrow X$ with $U'$ in $\U$ factors through $f$. The subcategory $\U$ is called \emph{contravariantly finite} if every object in $\C$ admits a right $\U$-approximation. The notions of \emph{left $\U$-approximations} and \emph{covariantly finite} subcategories are defined dually. A subcategory is \emph{functorially finite} if it is both covariantly and contravariantly finite.

Given a subcategory $\U$ of some abelian or triangulated category $\C$, we associate the subcategories
\begin{align*}
\Ud&=\{X\in \C \mid \Ext_{\C}^i(\U,X)=0 ~\text{ for }~ 1\leq i \leq d-1\}\\
\dU&=\{X\in \C \mid \Ext_{\C}^i(X,\U)=0 ~\text{ for }~ 1\leq i \leq d-1\}.
\end{align*}
Note that we write $\Ext_{\C}^i(X,Y)=\Hom_{\C}(X,Y[i])$ in the triangulated case. Using this notation, we give the definition of a $d$-cluster tilting subcategory. Recall that in the case where $\C$ is abelian, our subcategory $\U$ is called \emph{generating} (resp.\ \emph{cogenerating}) if for each object $X$ in $\C$ there exists an epimorphism $U \twoheadrightarrow X$ (resp.\ monomorphism $X \rightarrowtail U$) with $U$ in $\U$.

\begin{definition}[see \cites{Iyama,J, KR}] \label{def:d-CT}
A functorially finite subcategory $\U$ of an abelian or triangulated category $\C$ is \emph{$d$-cluster tilting} if it is generating-cogenerating (in the abelian case) and
\[
\U=\Ud=\dU.
\]
\end{definition}

It follows immediately from the definition that any $d$-cluster tilting subcategory necessarily contains all projective and all injective objects. In particular, a $d$-cluster tilting subcategory is automatically generating-cogenerating when the ambient category is a module category.

Following \cite{JK}, a finite dimensional algebra $\Lambda$ is called \emph{weakly $d$-representation finite} if it has a $d$-cluster tilting $\Lambda$-module. This means that there is a $\Lambda$-module $M$ such that $\add(M)$ is a $d$-cluster tilting subcategory of $\bmod \Lambda$. A weakly $d$-representation finite algebra is called \emph{$d$-representation finite $d$-hereditary} if the global dimension of $\Lambda$ is at most $d$. The reader should note that terminology related to higher representation finiteness varies in the literature. For instance, a $d$-representation finite $d$-hereditary algebra is in many papers known simply as \emph{$d$-representation finite}, see for example \cites{IO,HI,HI11,HIO,IO13,HS}. 

The module category of a $d$-representation finite $d$-hereditary algebra $\Lambda$ contains a unique $d$-cluster tilting subcategory $\U \subseteq \bmod \Lambda$ \cite{Iyama}*{Theorem 1.6}. In this situation, the subcategory
\[
  \U[d\Z] =  \add\left \{X[di] \mid X\in\U ~\text{ and }~ i\in\mathbb Z \right\} \subseteq \D^b(\Lambda)
\]
is a $d$-cluster tilting subcategory of the bounded derived category \cite{Iyama}*{Theorem 1.23}. This subcategory is closed under $[d]$, and thus yields an example of a $(d+2)$-angulated category \cite{GKO}.

\begin{example}\label{ex:d-CT}
Consider $\Lambda=KA_3/J^2$, where $A_3$ is the quiver $1 \rightarrow 2 \rightarrow 3$ and $J$ denotes the arrow ideal. The algebra $\Lambda$ is $2$-representation finite $2$-hereditary. \cref{fig:modA3} shows the AR-quiver of $\bmod{\Lambda}$ with the unique $2$-cluster tilting subcategory $\U =\add M$, where $M$ is given by the direct sum of all indecomposable projectives and injectives. The lift $\U[2\Z]$ to $\D^b(\Lambda)$ is shown in \cref{fig:derA3}.

\end{example}
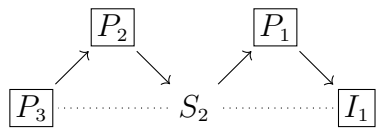
\begin{figure}[thb]
    \centering
    \begin{tikzpicture}[scale=1.07, ->,inner sep=2pt, outer sep=2pt]
    	\node[draw] (P3) at (0,0) {$P_3$};
	    \node[draw] (P2) at (1,1) {$P_2$};
    	\node (S2) at (2,0) {$S_2$};
    	\node[draw] (P1) at (3,1) {$P_1$};
    	\node[draw] (I1) at (4,0) {$I_1$};
    	\draw (P3) -- (P2);
    	\draw (P2) -- (S2);
	    \draw (S2) -- (P1);
	    \draw (P1) -- (I1);
	    \draw[dotted,-] (I1) -- (S2);
	    \draw[dotted,-] (S2) -- (P3);
    \end{tikzpicture}
    \caption{The AR-quiver of $\bmod\Lambda$ with rectangles around the indecomposable objects in the $2$-cluster tilting subcategory $\U$.}
    \label{fig:modA3}
\end{figure}
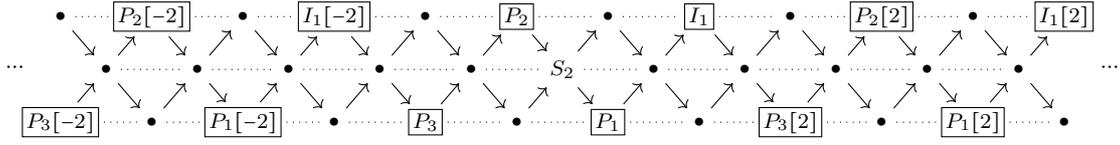
\begin{figure}[bht]
    \centering

\begin{tikzpicture}[font=\tiny, xscale=.6, yscale=.7, inner sep=1.5pt, outer sep=2pt]
		\node at (-1,1) {$\cdots$} ;
		\node at (23,1) {$\cdots$} ;
		
		\node[draw] (000) at (0,0) {$P_3[-2]$};
		\node (020) at (2,0) {$\bullet$};
		\node[draw] (040) at (4,0) {$P_1[-2]$};
		\node (060) at (6,0) {$\bullet$};
		\node[draw] (080) at (8,0) {$P_3$};
		\node (100) at (10,0) {$\bullet$};
		\node[draw] (120) at (12,0) {$P_1$};
		\node (140) at (14,0) {$\bullet$};
		\node[draw] (160) at (16,0) {$P_3[2]$};
		\node (180) at (18,0) {$\bullet$};
		\node[draw] (200) at (20,0) {$P_1[2]$};
		\node (220) at (22,0) {$\bullet$};

		\node (001) at (1,1) {$\bullet$};
		\node (021) at (3,1) {$\bullet$};
		\node (041) at (5,1) {$\bullet$};
		\node (061) at (7,1) {$\bullet$};
		\node (081) at (9,1) {$\bullet$};
		\node (101) at (11,1) {$S_2$};
		\node (121) at (13,1) {$\bullet$};
		\node (141) at (15,1) {$\bullet$};
		\node (161) at (17,1) {$\bullet$};
		\node (181) at (19,1) {$\bullet$};
		\node (201) at (21,1) {$\bullet$};
		
		\node (002) at (0,2) {$\bullet$};
		\node[draw] (022) at (2,2) {$P_2[-2]$};
		\node (042) at (4,2) {$\bullet$};
		\node[draw] (062) at (6,2) {$I_1[-2]$};
		\node (082) at (8,2) {$\bullet$};
		\node[draw] (102) at (10,2) {$P_2$};
		\node (122) at (12,2) {$\bullet$};
		\node[draw] (142) at (14,2) {$I_1$};
		\node (162) at (16,2) {$\bullet$};
		\node[draw] (182) at (18,2) {$P_2[2]$};
		\node (202) at (20,2) {$\bullet$};
		\node[draw] (222) at (22,2) {$I_1[2]$};

		\begin{scope}[->]
			\draw (000)--(001);
			\draw (020)--(021);
			\draw (040)--(041);
			\draw (060)--(061);
			\draw (080)--(081);
			\draw (100)--(101);
			\draw (120)--(121);
			\draw (140)--(141);
			\draw (160)--(161);
			\draw (180)--(181);
			\draw (200)--(201);
			
			\draw (002)--(001);
			\draw (022)--(021);
			\draw (042)--(041);
			\draw (062)--(061);
			\draw (082)--(081);
			\draw (102)--(101);
			\draw (122)--(121);
			\draw (142)--(141);
			\draw (162)--(161);
			\draw (182)--(181);
			\draw (202)--(201);
			
			\draw (001)--(020);
			\draw (021)--(040);
			\draw (041)--(060);
			\draw (061)--(080);
			\draw (081)--(100);
			\draw (101)--(120);
			\draw (121)--(140);
			\draw (141)--(160);
			\draw (161)--(180);
			\draw (181)--(200);
			\draw (201)--(220);
			
			\draw (001)--(022);
			\draw (021)--(042);
			\draw (041)--(062);
			\draw (061)--(082);
			\draw (081)--(102);
			\draw (101)--(122);
			\draw (121)--(142);
			\draw (141)--(162);
			\draw (161)--(182);
			\draw (181)--(202);
			\draw (201)--(222);
		\end{scope}
		
		\begin{scope}[dotted]
			\draw (000)--(020);
			\draw (020)--(040);
			\draw (040)--(060);
			\draw (060)--(080);
			\draw (080)--(100);
			\draw (100)--(120);
			\draw (120)--(140);
			\draw (140)--(160);
			\draw (160)--(180);
			\draw (180)--(200);
			\draw (200)--(220);
			
			\draw (001)--(021);
			\draw (021)--(041);
			\draw (041)--(061);
			\draw (061)--(081);
			\draw (081)--(101);
			\draw (101)--(121);
			\draw (121)--(141);
			\draw (141)--(161);
			\draw (161)--(181);
			\draw (181)--(201);
			
			\draw (002)--(022);
			\draw (022)--(042);
			\draw (042)--(062);
			\draw (062)--(082);
			\draw (082)--(102);
			\draw (102)--(122);
			\draw (122)--(142);
			\draw (142)--(162);
			\draw (162)--(182);
			\draw (182)--(202);
			\draw (202)--(222);
		\end{scope}
	\end{tikzpicture}
    \caption{The AR-quiver of $\D^b (\Lambda)$ with rectangles around the indecomposable objects in the $2$-cluster tilting subcategory $\U[2\Z]$. Note that we identify a $\Lambda$-module with its associated stalk complex concentrated in degree $0$.} 
    \label{fig:derA3}
\end{figure}

\subsection{Gentle algebras}
Gentle algebras constitute a large and well-studied class of algebras. They first appeared as iterated tilted algebras of Dynkin type $A$ \cites{AH} and $\widetilde A$ \cite{AS}, and can be seen as generalizations of algebras of Dynkin type $A$.

\begin{definition}\label{def:gentle}
An algebra of the form $KQ/I$ is called \emph{gentle} if the following conditions hold:
\begin{enumerate}
    \item The quiver $Q$ is finite;
    \item For all $\alpha \in Q_1$, there exists at most one arrow $\beta$ such that $\alpha \beta \notin I$ and at most one arrow $\gamma$ such that $\gamma \alpha \notin I$;
    \item For all $\alpha \in Q_1$, there exists at most one arrow $\beta$ with $t(\alpha) = s(\beta)$ such that $\alpha \beta \in I$ and at most one arrow $\gamma$ with $t(\gamma) = s(\alpha) $ such that $\gamma \alpha \in I$;
    \item The ideal $I$ is admissible and generated by the relations in (3). 
\end{enumerate}
\end{definition}

Note that part (2) and (3) in \cref{def:gentle} imply that there are at most two incoming and at most two outgoing arrows at each vertex in the quiver $Q$. Gentle algebras are special biserial and their indecomposable modules have been classified in terms of string and band combinatorics \cite{BR}. In this paper we only need to work with string modules, so we briefly recall their definition here. 

For every arrow $\alpha$  in $Q_1$ with $s(\alpha) = x$ and $t(\alpha) = y$, we define its formal inverse  $\alpha^{-1}$ by setting $s(\alpha^{-1}) = y$ and $t(\alpha^{-1}) = x$. A \emph{walk} is a sequence $\alpha_{1} \dots \alpha_{r}$ of arrows and inverse arrows such that $t(\alpha_{i}) = s(\alpha_{i + 1})$ and $\alpha_{i + 1} \neq \alpha_{i}^{-1}$ for all $ 1 \leq i \leq  r - 1$. A \emph{string} is a walk $w$ in $Q$ such that no subword of $w$ or of $w^{-1}$ is in $I$. If $w$ is a string, the associated \emph{string module} $M(w)$ is given by the quiver representation obtained by replacing every vertex in $w$  by a copy of $K$ and every arrow by the identity map. 

\subsection{The geometric model} \label{sec: geometric model}
In \cite{OPS}, building on \cite{Schroll15} and \cite{Schroll18}, a geometric model for the derived category of a gentle algebra is given in terms of surface dissections. It is closely related to the partially wrapped Fukaya category of surfaces with stops described in \cite{HKK} and further studied in \cite{LP}. We give a brief introduction to the geometric model, emphasizing aspects that are needed in \cref{section: derived category of gentle algebra}. The reader is referred to \cite{OPS} for more detailed explanations.

The geometric model is based on a bijection between gentle algebras and certain surface dissections as described in \cite{OPS}, see also \cites{BC-S, PPP}. The construction is as follows. Consider a pair $(S,M)$, where $S$ is a compact oriented surface with boundary and $M$ is a set of marked points on the boundary. Let $\Gamma$ be a \emph{dissection} of $(S,M)$ into polygons. That is, the vertices of $\Gamma$ are exactly the marked points in $M$ and the complement of $\Gamma$ in $S$ is a disjoint union of polygons. We call the dissection $\Gamma$ \emph{admissible} if each polygon either has exactly one boundary segment or, if it has no boundary segment, encloses a boundary component with no marked points on it.

Given an admissible dissection $\Gamma$ of $(S,M)$, we describe how to obtain a quiver $Q$. The vertices of $Q$ are in bijection with the edges of $\Gamma$. If two edges $\gamma$ and $\gamma'$ of $\Gamma$ are incident with the same vertex of $\Gamma$ such that $\gamma'$ directly follows $\gamma$ in the orientation of the surface, there is an arrow from $\gamma$ to $\gamma'$ in $Q$. We define an ideal of relations $I$ of $KQ$ as follows. Suppose $\alpha$ and $\beta$ are two composable arrows in $Q$. If the edges of $\Gamma$ corresponding to the vertices $s(\alpha), s(\beta)$ and $t(\beta)$ in $Q$ are incident with the same vertex of $\Gamma$ and directly follow each other in the orientation of $S$, then $\alpha \beta \notin I$. Otherwise, we set $\alpha \beta \in I$. The resulting algebra $\Lambda=KQ/I$ is gentle, and every gentle algebra arises in this way. 

\begin{example}\label{ex:quiver from surface}
Consider the gentle algebra $\Lambda=KA_n/J^2$, where $A_n$ is the linearly oriented quiver of Dynkin type $A$ with $n$ vertices and $J$ is the arrow ideal. This algebra arises from a dissection of the disk with $n+1$ marked points on the boundary, as illustrated in \cref{fig:quiver from surface}. 
\end{example}

\begin{figure}[htb]
    \centering    \begin{tikzpicture}[radius=5cm, scale=.55]
	\coordinate (dotted_start) at (350:5);
	\draw (dotted_start) arc[start angle= 350, delta angle= 270];
	\draw[dashed] (dotted_start) arc[start angle=350, delta angle=-90];
	
	\coordinate (m3) at (220:5);
	\coordinate (m2) at (165:5);
	\coordinate (m1) at (110:5);
	\coordinate (mt) at (80:5);
	\coordinate (m-) at (10:5);
	\filldraw (m3) circle[radius=.1];
	\filldraw (m2) circle[radius=.1];
	\filldraw (m1) circle[radius=.1];
	\filldraw (mt) circle[radius=.1];
	\filldraw (m-) circle[radius=.1];
	
	\draw (m2) to[bend right] node[midway] (p1) {$\bullet$} (m1); 
	\draw (m3) to[bend right] node[midway] (p2) {$\bullet$} (m2);
	\draw[dashed] (m3) to[bend left] node[pos=0.6] (p3) {$\bullet$} (0,-4);
	\draw[dashed] (m-) to[bend right] node[pos=0.6] (p-) {$\bullet$} (3,-3);
	\draw (mt) to[bend right] node[midway] (pt) {$\bullet$} (m-);
	
	\filldraw[petal] (p1) circle[radius=.1] node[above left] {$1$};
	\filldraw[petal] (p2) circle[radius=.1] node[left] {$2$};
	\filldraw[petal] (p3) circle[radius=.1];
	\filldraw[petal] (p-) circle[radius=.1];
	\filldraw[petal] (pt) circle[radius=.1] node[above right] {$n$};
	
	\draw[petal, ->] (p1) to[bend left] coordinate[pos=0.8] (r1) (p2);
	\draw[petal, ->] (p2) to[bend left] coordinate[pos=0.2] (r2) (p3);
	\draw[petal, ->] (p-) to[bend left] (pt);
	
	\draw[dotted, petal] (r1) to[bend left] (r2);
\end{tikzpicture}

    \caption{The geometric model associated to the gentle algebra $KA_n/J^2$. The dissection $\Gamma$ is given by the thinner outer curves, while the quiver $A_n$ is drawn with thicker blue
    arrows. The relations are indicated by the dotted line.}
    \label{fig:quiver from surface}
\end{figure}
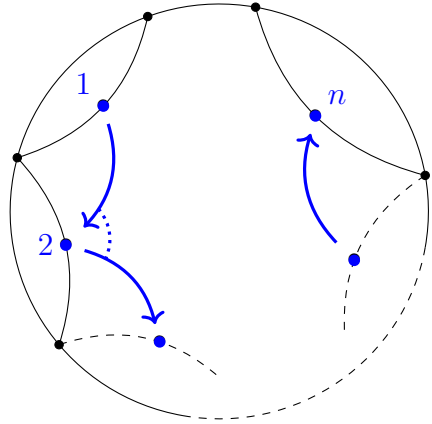

Given a surface $(S,M)$ with admissible dissection $\Gamma$, we define the dual (embedded) graph $L$ of $\Gamma$ as follows. The vertices of $L$ lie on the boundary of $S$ in such a way that on each boundary component, the vertices of $L$ and $\Gamma$ alternate. If a boundary component does not contain any marked points of $M$, then we replace it by a vertex of $L$ and refer to it as a \emph{puncture}. The edges of the dual graph $L$ are in bijection with the edges of $\Gamma$, and for each edge $\gamma$ in $\Gamma$ there exists a unique edge $\ell$ in $L$ crossing $\gamma$ and crossing no other edge of $\Gamma$. The dual graph $L$ gives an admissible dissection of $(S,M)$.

\begin{example}
\label{ex:dual graph}
The dual graph $L$ of the dissection $\Gamma$ given in \cref{ex:quiver from surface} is shown in \cref{fig:dual graph}.
\end{example}

\begin{figure}[hbt]
    \centering
\begin{tikzpicture}[radius=5cm,scale=.55]
	\coordinate (dotted_start) at (350:5);
	\draw (dotted_start) arc[start angle= 350, delta angle= 270];
	\draw[dashed] (dotted_start) arc[start angle=350, delta angle=-90];
	\coordinate (m3) at (210:5);
	\coordinate (m2) at (165:5);
	\coordinate (m1) at (110:5);
	\coordinate (mt) at (80:5);
	\coordinate (m-) at (10:5);
	\filldraw (m3) circle[radius=.1];
	\filldraw (m2) circle[radius=.1];
	\filldraw (m1) circle[radius=.1];
	\filldraw (mt) circle[radius=.1];
	\filldraw (m-) circle[radius=.1];
	
	\coordinate (l3) at (230:5);
	\coordinate (l2) at (192.5:5);
	\coordinate (l1) at (137.5:5);
	\coordinate (lt) at (45:5);
	\coordinate (l-) at (355:5);
	\coordinate (l) at (95:5);
	\filldraw[laminate] (l3) circle[radius=.1];
	\filldraw[laminate] (l2) circle[radius=.1];
	\filldraw[laminate] (l1) circle[radius=.1];
	\filldraw[laminate] (lt) circle[radius=.1];
	\filldraw[laminate] (l-) circle[radius=.1];
	\filldraw[laminate] (l) circle[radius=.1];
	
	\draw (m2) to[bend right]  (m1); 
	\draw (m3) to[bend right]  (m2);
	\draw[dashed] (m3) to[bend left] (0,-4);
	\draw[dashed] (m-) to[bend right] (3,-3);
	\draw (mt) to[bend right] (m-);
	
	\draw[laminate] (l) to[bend left] (l1);
	\draw[laminate] (l) to[bend left] (l2);
	\draw[laminate,dashed] (l) to[bend left] (l3);
	\draw[laminate,dashed] (l) to[bend right] (l-);
	\draw[laminate] (l) to[bend right] (lt);
	
	\end{tikzpicture}

    \caption{The dual graph (in red) of an admissible dissection of the disk.}
    \label{fig:dual graph}
\end{figure}
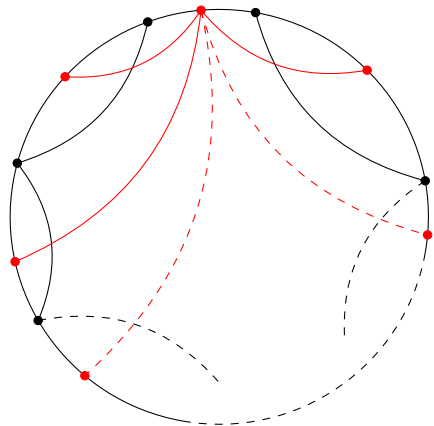

We let the notion \emph{closed curve} refer to a homotopy class of closed curves in $(S,M)$. Similarly, an \emph{arc} is a homotopy class of curves between marked points or wrapping around punctures on one or both ends (following the orientation of $S$). By abuse of notation, we typically let a representative $\gamma$ of an arc also denote the arc itself. If the underlying curve of an arc connects two marked points, we call the arc \emph{finite}. A finite arc is called \emph{minimal} if it is homotopy equivalent to a segment of the boundary with no marked points. In particular, any minimal arc has both endpoints on the same boundary component.

Let $\gamma$ be an arc or a closed curve in $(S, M)$. A \emph{grading} of $\gamma$ is a function $f\colon L \cap \gamma \to \mathbb{Z}$ subject to the condition described below. Let $p_1$ and $p_2$ be two consecutive (in the orientation of $\gamma$) intersection points of $\gamma$ with $L$. The points $p_1$ and $p_2$ correspond to edges $l_1$ and $l_2$ of a unique polygon $P$ of the dissection given by $L$ in such a way that $\gamma$ enters $P$ via $l_1$ and leaves via $l_2$. The polygon $P$ has exactly one boundary segment. The grading $f$ satisfies $f(p_2) = f(p_1) +1$ if this boundary segment lies to the left of $\gamma$ with respect to the orientation of $\gamma$ and $f(p_2) = f(p_1) - 1$ otherwise. A grading $f$ of $\gamma$ is hence determined by its value $f(p)$ at one point $p \in L \cap \gamma$. Note that in the case where $\gamma$ is a closed curve, there need not exist a grading of $\gamma$. In the case where a grading $f$ of $\gamma$ does exist, we refer to the pair $(\gamma, f)$ as a \emph{graded arc} or a \emph{graded closed curve}.

By \cite{BM} the indecomposable objects in $\D^b(\Lambda)$ are in bijection with so-called graded homotopy strings and bands. This bijection allows us to divide the indecomposable objects into \emph{string objects} and \emph{band objects}. The graded homotopy strings are in bijection with graded arcs. Consequently, graded arcs are in bijection with string objects in $\D^b(\Lambda)$. Band objects occur in one parameter families, which are in bijection with graded closed curves. Altogether, this yields a correspondence between indecomposable objects in $\D^b(\Lambda)$ and graded arcs and graded closed curves in $(S,M)$. Given a string object $X$ in $\D^b(\Lambda)$, we use the notation $(\gamma_X,f_X)$ for the corresponding graded arc. When we do not need to describe the grading explicitly, we refer to the graded arc simply by $\gamma_X$. 

\begin{figure}[hbt]
    \centering
    \begin{tikzpicture}[radius=5cm, scale=.75]
	\draw (0,0) circle;
	\coordinate (m4) at (30:5);
	\coordinate (m1) at (150:5);
	\coordinate (m2) at (230:5);
	\coordinate (m3) at (310:5);
	
	\coordinate (lt) at (90:5);
	\coordinate (l1) at (350:5);
	\coordinate (l2) at (270:5);
	\coordinate (l3) at (190:5);
	\filldraw[laminate] (l1) circle[radius=.1];
	\filldraw[laminate] (l2) circle[radius=.1];
	\filldraw[laminate] (l3) circle[radius=.1];
	\filldraw[laminate] (lt) circle[radius=.1];
	\draw[laminate] (lt) to[bend right] node[above right, pos=0.6] {$l_3$} (l1);
	\draw[laminate] (lt) to 			node[right, pos=0.6] {$l_2$} (l2);
	\draw[laminate] (lt) to[bend left]  node[below right, pos=0.6] {$l_1$}(l3);
	
	\draw (m1) to[bend left] node[right, pos=0.6]{$\gamma_1$} (m2); 
	\draw (m2) to[bend left] node[above, pos=0.7]{$\gamma_2$}  (m3);
	\draw (m3) to[bend left] node[left, pos=0.4]{$\gamma_3$}  (m4);
	
	\draw[very thick, petal] (m1) to node[above, pos=0.15] {$\gamma_X$}
		node[below left, pos=0.39]{$p_1$} node[below left, pos=0.51]{$p_2$} node[below left, pos=0.63]{$p_3$} (m4);
	
	\filldraw (m1) circle[radius=.1];
	\filldraw (m2) circle[radius=.1];
	\filldraw (m3) circle[radius=.1];
	\filldraw (m4) circle[radius=.1];
	
	\end{tikzpicture}
    \caption{The geometric model of $KA_3/J^2$, with the dual graph $L$ given by the curves $l_1, l_2$ and $l_3$. Objects corresponding to minimal graded arcs are described in \cref{ex:indecomposable object}.}
    \label{fig:indecomposable object}
\end{figure}
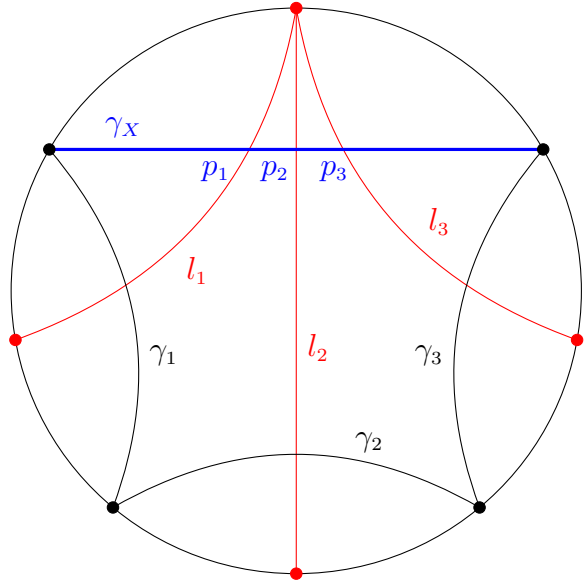

\begin{example}\label{ex:indecomposable object}
We illustrate how graded arcs in the surface of our running example correspond to indecomposable objects in the derived category. In \cref{fig:indecomposable object} we display the geometric model of $\Lambda=KA_3/J^2$, where the dual graph $L$ is drawn in red. Recall that we label the vertices in $A_3$ by $1 \rightarrow 2 \rightarrow 3$. Edges of the dual graph are denoted $l_i$ for $i\in\{1,2,3\}$. The arcs corresponding to indecomposable projectives up to shift are drawn in black, so that the arc $\gamma_i$ corresponding to the stalk complex of the projective in $i$ crosses $l_i$. 

Consider the blue arc $\gamma_X$ which intersects all three edges of $L$. The intersection of $\gamma_X$ and $l_i$ is labeled $p_i$. As the grading of an arc is determined by its grading in one intersection point, we choose $f_X(p_1)=0$, which gives $f_X(p_2)=-1$ and $f_X(p_3)=-2$. Thus, by the description in \cite{OPS}, the indecomposable object $X$ in $\D^b(\Lambda$) corresponding to the graded arc $(\gamma_X,f_X)$ is given by 
$$\cdots \rightarrow 0\rightarrow \overset{-2}{P_3} \rightarrow \overset{-1}{P_2}\rightarrow \overset{0}{P_1}\rightarrow 0 \rightarrow \cdots,$$ 
where the grading is written above the complex. We recognize this as the projective resolution of the simple in $1$, so $X$ is isomorphic in $\D^b(\Lambda)$ to the stalk complex with the simple in $1$ in degree $0$.
\end{example}

In \cref{section: derived category of gentle algebra} we also consider the perfect derived category of a gentle algebra $\Lambda$. This is the full isomorphism closed subcategory $\D^b(\proj\Lambda) \subseteq \D^b(\Lambda)$ consisting of bounded complexes of finitely generated projective $\Lambda$-modules. The indecomposable perfect objects correspond to finite graded arcs and graded closed curves. If $\Lambda$ has finite global dimension, the associated surface has no punctures and the perfect derived category is equivalent to the bounded derived category.

By \cite{Happel}, the perfect derived category of a gentle algebra has AR-triangles. The AR-translate of an indecomposable object $X$ corresponding to a finite graded arc $(\gamma_X,f_X)$ can be computed in terms of the geometric model. More precisely, the AR-translate $\tau X$ corresponds to the arc $\gamma_{\tau X}$ obtained by moving the endpoints of $\gamma_X$ to the next marked points on the boundary (following the orientation of $S$) equipped with a suitable grading.

Given two homotopy classes of graded curves $(\gamma_X, f_X)$ and $(\gamma_Y, f_Y)$ corresponding to indecomposable objects $X$ and $Y$ in $\D^b(\Lambda)$, there is an explicit bijection between the oriented graded intersections of $\gamma_X$ with  $\gamma_Y$ and a basis of  $\Hom_{\D^b(\Lambda)}(X,Y)$. Note that when considering intersections of homotopy classes of curves, we always choose representatives such that the number of intersections is minimal. We give a summary of the definition of an oriented graded intersection using \cref{fig:intersections} and \cref{fig:intersections_puncture}, and refer the reader to \cite{OPS}*{Definition 3.7} for more details.

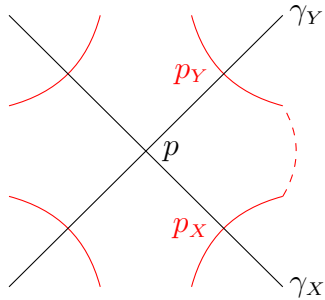
\begin{figure}[bht]
    \pgfmathsetmacro{\dist}{3}
    \pgfmathsetmacro{\nodedisp}{3pt}
    \centering
    \begin{tikzpicture}[scale=1.2, circle, inner sep=1pt]
	    \coordinate (bottomleft) at (0,0);
	    \coordinate (topleft) at (0,\dist);
	    \coordinate (bottomright) at (\dist,0);
	    \coordinate (topright) at (\dist,\dist);
	
	    \path[laminate] (bottomleft) edge[white] +(0,1) +(0,1)  edge [laminate, bend left] +(1,0);
	    \path[laminate] (topleft) edge[white] +(0,-1) +(0,-1)  edge [laminate, bend right] +(1,0);
	    \path[laminate] (bottomright) edge[white] +(0,1) 
					+(0,1)  edge [laminate, bend right] node[left=\nodedisp] {$p_X$} +(-1,0) coordinate (bottomrighttop);
	    \path[laminate]	(topright) edge[white] +(0,-1) +(0,-1)  edge [laminate, bend left] node[left=\nodedisp] {$p_Y$} +(-1,0) coordinate (toprightbottom);
	    \draw[laminate,dashed, bend right] (bottomrighttop) to (toprightbottom);
	
	    \draw (bottomleft)-- node[right=\nodedisp] {$p$} node[at end, right] {$\gamma_Y$} (topright);
	    \draw (topleft)-- node[at end, right] {$\gamma_X$} (bottomright);
	\end{tikzpicture}
    \caption{Illustration of oriented graded intersection from $(\gamma_X,f_X)$ to $(\gamma_Y,f_Y)$. Intersection points of $\gamma_X$ and $\gamma_Y$ with edges of the polygon given by $L$ as a dissection of the surface are labeled by $p_X$ and $p_Y$, as indicated.}
    \label{fig:intersections}
\end{figure}

In the first case, let $\gamma_X$ and $\gamma_Y$ intersect such that the intersection point $p$ is not at a puncture. The intersection lies in one of the polygons given by the dual graph as a dissection of the surface, and this polygon has exactly one boundary segment. In \cref{fig:intersections} we assume that this boundary segment does not lie between $p_X$ and $p_Y$, as indicated by the red dotted line. Note that the intersection point $p$ can be on the boundary of $S$ and that the edges of the dual graph are not necessarily distinct. Given a situation as in \cref{fig:intersections}, there is an oriented graded intersection from $(\gamma_X,f_X)$ to $(\gamma_Y,f_Y)$, and hence a non-zero morphism from $X$ to $Y$ in the derived category, whenever $f_X(p_X)=f_Y(p_Y)$.

\begin{figure}[htb]
    \centering

    \begin{tikzpicture}[scale=.6, circle, inner sep=1pt]
    \draw [domain=0:25.5,variable=\t,smooth,samples=100] plot ({-\t r}: {0.0076*\t*\t}) node[left]{$\gamma_X$};
    \draw [domain=0:25.5,variable=\t,smooth,samples=100] plot ({-\t r}: {0.0072*(\t+3)*(\t+3)})node[right]{$\gamma_Y$};
    \filldraw[laminate] (0,0) circle[radius=.1];
    \draw[laminate] (0,0) -- coordinate[pos=0.65](px) coordinate[pos=0.92](py) node[at end,right]{$l$}( 40:7);
    \draw[laminate, dashed] (0,0) -- (80:6.3);
    \draw[laminate, dashed] (0,0) -- (10:6.5);
    \draw[laminate, dashed] (0,0) -- (220:5.5);
    \node[laminate] at ([shift={(0,-0.6)}]px) {$p_X$};
    \node[laminate] at ([shift={(0,-.6)}]py) {$p_Y$};

	\end{tikzpicture}
	
    \caption{Illustration of oriented graded intersection from $(\gamma_X,f_X)$ to $(\gamma_Y,f_Y)$ in the case where $\gamma_X$ and $\gamma_Y$ are infinite arcs wrapping around the same puncture at one end.} 
    
    \label{fig:intersections_puncture}
\end{figure}
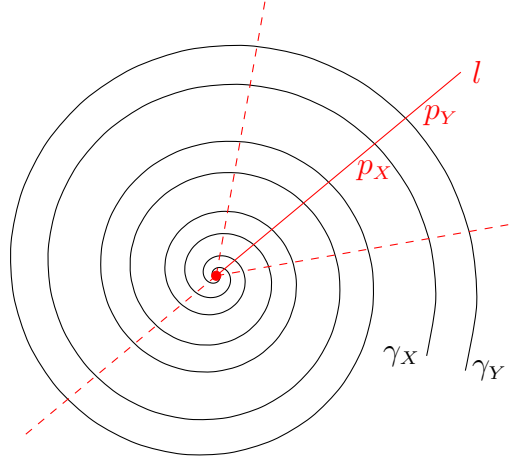

In the second case, assume that $\gamma_X$ and $\gamma_Y$ are infinite arcs, wrapping around the same puncture $p$ at one end. We then say that the arcs intersect at $p$ (even though they technically only approach the point of intersection asymptotically). Consider an edge $l$ in the dual graph $L$ ending at $p$. Let $p_X$ and $p_Y$ be intersection points of $l$ with $\gamma_X$ and $\gamma_Y$, respectively. The situation is shown in \cref{fig:intersections_puncture}. In this setup, there is an oriented graded intersection from  $\gamma_X$ to $\gamma_Y$ if $f_X(p_X)=f_Y(p_Y)$.

The possibilities for an intersection point $p$ of graded curves $\gamma_X$ and $\gamma_Y$ to give rise to a non-zero morphism in the derived category can be summarized as follows:

\begin{itemize}
    \item If $p$ corresponds to a point in the interior of the surface and $f_X$ is any grading of $\gamma_X$, then there exists a unique grading $f_Y$ of $\gamma_Y$ such that $p$ corresponds to an oriented graded intersection from $(\gamma_X, f_X)$ to $(\gamma_Y, f_Y)$ and to an oriented graded intersection from $(\gamma_Y, f_Y)$ to $(\gamma_X, f_X[1])$.
    \item If $p$ is on the boundary of $S$ and  $f_X$ is any  grading of $\gamma_X$, then there exists a unique grading $f_Y$ of $\gamma_Y$ such that $p$ corresponds to either an oriented graded intersection from $(\gamma_X, f_X)$ to $(\gamma_Y, f_Y)$ or an oriented graded intersection from $(\gamma_Y, f_Y)$ to $(\gamma_X, f_X)$.
    \item It remains to consider the case where $p$ is a puncture. For this, let $w$ denote the number of endpoints of edges of $L$ that are incident with $p$. If $f_X$ is any  grading of $\gamma_X$, then there exists a grading $f_Y$ of $\gamma_Y$ such that $p$ corresponds to a family of oriented graded intersections from $(\gamma_X,f_X)$ to $(\gamma_Y,f_Y[mw])$ and a family of oriented graded intersections from $(\gamma_Y,f_Y)$ to $(\gamma_X,f_X[(m+1)w])$ for $m\geq 0$.
\end{itemize}

\begin{example}\label{ex:morphism}
We illustrate the correspondence between morphisms and oriented graded intersections in our running example. Consider the arcs $\gamma_X$ and $\gamma_3$ from \cref{ex:indecomposable object}, see \cref{fig:indecomposable object}. Let $f_X$ be the grading of $\gamma_X$ that is described in \cref{ex:indecomposable object} and define a grading $f_3$ on $\gamma_3$ by $f_3(p)=-2$ for the sole intersection point $p \in L \cap \gamma_3$. The graded arc $(\gamma_3,f_3)$ corresponds to the object $P_3[2]$, by which we mean the stalk complex with the projective $P_3$ in degree $-2$. As $f_X(p_3) = f_3(p)$, there is an oriented graded intersection from $(\gamma_X,f_X)$ to $(\gamma_3,f_3)$. This corresponds to a non-zero morphism from $X$ to $P_3[2]$ in the derived category. 
\end{example}

\section{Weakly \texorpdfstring{$d$}{d}-representation finite gentle algebras}
\label{section: module category of gentle algebra}
In this section we give a complete classification of the weakly $d$-representation finite gentle algebras, as well as $d$-representation finite $d$-hereditary gentle algebras.

We denote the quivers
\[
\begin{tikzcd}[column sep=12, row sep=16]
&&&&&& 0 \arrow[rd, bend left] \\
1\arrow[r] & 2\arrow[r] &\cdots\arrow[r] &n& \text{and} &n \arrow[ru, bend left] & & 1 \arrow[ld, bend left] \\
&&&&&& \cdots \arrow[lu, bend left]
\end{tikzcd}
\]
by $A_n$ and $\widetilde{A}_n$, respectively. A \emph{Nakayama algebra} is a path algebra of one of these quivers modulo an admissible ideal, see for instance \cite{ASS}*{Chapter V} for more details. Recall that such an algebra is called \emph{radical square zero} if the admissible ideal is given by $J^2$, where $J$ denotes the arrow ideal. In our next result, we show that if $\Lambda$ is a gentle algebra and $\bmod \Lambda$ admits a $d$-cluster tilting subcategory for some $d \geq 2$, then $\Lambda$ is a radical square zero Nakayama algebra. In particular, this yields that the only examples of $d$-cluster tilting subcategories of module categories that arise from gentle algebras are the ones known from \cites{DI,Vaso}.

\begin{theorem} \label{thm: nakayama}
Let $\Lambda$ be a gentle algebra. If $\bmod \Lambda$ contains a $d$-cluster tilting subcategory for some $d \geq 2$, then $\Lambda$ is a radical square zero Nakayama algebra.
\end{theorem}

\begin{proof}
Any vertex $x$ in the quiver of the gentle algebra $\Lambda$
is part of a subquiver  
\begin{center}
\begin{tikzpicture}[yscale=.5, xscale=1.2]
\node (x) at (0,0) {$x$};
\node (a) at (-1,1) {$a$};
\node (b) at (-1,-1) {$b$};
\node (c) at (1,1) {$c$};
\node (d) at (1,-1) {$d$};

\draw[->] (a)--node[above]{$\alpha$}(x);
\draw[->] (b)--node[below]{$\beta$}(x);
\draw[->] (x)--node[above]{$\gamma$}(c);
\draw[->] (x)--node[below]{$\delta$}(d);

\draw[dotted] (-.4,.4) to[bend left] (.4,.4);
\draw[dotted] (-.4,-.4) to[bend right] (.4,-.4);
\end{tikzpicture}
\end{center}
with relations as indicated by the dotted lines. Note that we allow arrows to be non-existent. We write $\alpha = \varnothing$ in the case where the arrow $\alpha$ in the above figure does not exist.

The projective module $P_x$ associated to the vertex $x$ is represented by the string $u_\gamma^{-1}\gamma^{-1}\delta u_\delta$, where $u_\delta$ and $u_\gamma$ are the (possibly trivial) maximal strings such that $\delta u_\delta$ and $\gamma u_\gamma$ are non-zero strings. Similarly, the injective $I_x$ is represented by $v_\alpha\alpha\beta^{-1} v_\beta^{-1} $, where $v_\alpha$ and $v_\beta$ are maximal strings such that $v_\alpha \alpha$ and $v_\beta \beta$ are non-zero strings. By \cite{Schroer}, there is a short exact sequence
\[0 \to M(u_\gamma^{-1}\gamma^{-1}\delta u_\delta) \rightarrow M(v_\alpha\alpha\delta u_\delta) \oplus M(v_\beta\beta\gamma u_\gamma)\rightarrow M(v_\alpha\alpha\beta^{-1} v_\beta^{-1}) \to 0\]
starting in $P_x$ and ending in $I_x$. This sequence does not split as long as the following conditions are satisfied:
\begin{enumerate}[label=(\roman*)]
    \item If $\alpha = \varnothing$, then $\gamma \neq \varnothing$;
    \item If $\beta = \varnothing$, then $\delta \neq \varnothing$.
\end{enumerate}
In this case the non-split short exact sequence is known as an \emph{overlap extension} \cite{CPS}*{Definition 3.1}, see also \cite{BDMTY}, and we have $\Ext^1_{\Lambda}(I_x,P_x) \neq 0$.

Assume that $\bmod \Lambda$ contains a $d$-cluster tilting subcategory for some $d \geq 2$. As a $d$-cluster tilting subcategory necessarily contains all projective and injective modules, this implies that $\Ext^1_{\Lambda}(I_x,P_x) = 0$ for every vertex $x$. By our argument above, this means that for every vertex in the quiver of $\Lambda$, at least one of (i) or (ii) does not hold. Considering all possible configurations of the subquiver associated to a vertex $x$, we see that this only happens if there is exactly one arrow adjacent to $x$ or if we have a situation of the type
\begin{center}
\begin{tikzpicture}
\node (y) at (-2,0) {};
\node (x) at (0,0) {$x$};
\node (z) at (2,0) {};

\draw[->] (y)--node{} (x);
\draw[->] (x)--node{} (z);

\draw[dotted] (-0.5,0) to [out=90, in=90, distance=0.5cm] (0.5,0);
\end{tikzpicture}
\end{center}
with exactly one incoming and one outgoing arrow. This yields the result.
\end{proof}

\begin{remark}
Note that the result above does not hold in the more general case where $\Lambda$ is assumed to be a string algebra. For examples of this, see for instance \cite{Vaso21}.
\end{remark}

Combining \cref{thm: nakayama} with previously known classification results for Nakayama algebras \cites{DI,Vaso}, we obtain the following characterization of weakly $d$-representation finite gentle algebras. Recall that we denote the arrow ideal of a path algebra $KQ$ by $J$.

\begin{corollary}\label{cor:weakly d-RF gentle}
Let $\Lambda$ be a gentle algebra and assume $d \geq 2$. Then $\Lambda$ is weakly $d$-representation finite if and only if one of the following statements holds:
    \begin{enumerate}
    \item $\Lambda = KA_n/J^2~$ with $~n=dk+1~$ for some $k\geq 0$;
    \item $\Lambda = K\widetilde{A}_n/J^2~$ with $~n = dk-1~$ for some $k \geq 1$.
    \end{enumerate}
\end{corollary}

\begin{proof}
If one of the statements \emph{(1)} or \emph{(2)} holds, it follows from \cite{Vaso}*{Theorem 2} and \cite{DI}*{Theorem 5.1} that $\Lambda$ is weakly $d$-representation finite. For this, notice that existence of a $d$-cluster tilting subcategory implies existence of a $d$-cluster tilting module as our algebra is representation finite.

Assume next that $\Lambda$ is weakly $d$-representation finite. In particular, this means that $\bmod \Lambda$ contains a $d$-cluster tilting subcategory. By \cref{thm: nakayama}, this yields $~\Lambda = KA_n/J^2~$ or $~\Lambda = K\widetilde{A}_n/J^2~$ for some $n$. Our conclusion now follows by applying  \cite{Vaso}*{Theorem 2} and \cite{DI}*{Theorem 5.1} and using that $\Lambda$ has Loewy length $2$.
\end{proof}

Our approach also yields a classification of $d$-representation finite $d$-hereditary gentle algebras. Recall that a weakly $d$-representation finite algebra is called $d$-representation finite $d$-hereditary in the case where the global dimension is at most $d$.

\begin{corollary} \label{cor: d-RF gentle}
Let $\Lambda$ be a gentle algebra of global dimension $d$ with \mbox{$d \geq 2$}. Then $\Lambda$ is $d$-representation finite $d$-hereditary if and only if \mbox{$\Lambda = KA_n/J^2$} with $n=d+1$.
\end{corollary}

\begin{proof}This is an immediate consequence of \cref{thm: nakayama} combined with Vaso's classification of $d$-representation finite $d$-hereditary Nakayama algebras \cite{Vaso}*{Theorem 3}.
\end{proof}

In the case where $d\geq 3$ or the preprojective algebra of $\Lambda$ is a planar quiver with potential, \cref{cor: d-RF gentle} can be recovered by work of Sand{\o}y and Thibault \cite{ST}*{Theorem B}. However, their result does not cover all gentle algebras, since there exist gentle algebras with a non-planar quiver and global dimension $2$. One algebra of this type is shown in \cref{ex:K33}, and it gives rise to an infinite family of such algebras.

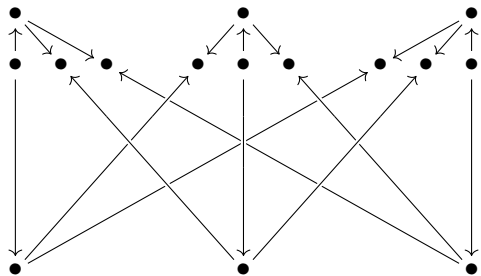
\begin{figure}[bht]
    \centering
    \begin{tikzpicture}[xscale=3,yscale=1.7, ->,circle, inner sep=1pt]
	\node (1) at (0,2) {$\bullet$};
	\node (1a) at (0,1.6) {$\bullet$};
	\node (1b) at (.2,1.6) {$\bullet$};
	\node (1c) at (.4,1.6) {$\bullet$};

	\node (2) at (1,2) {$\bullet$};
	\node (2a) at (.8,1.6) {$\bullet$};
	\node (2b) at (1,1.6) {$\bullet$};
	\node (2c) at (1.2,1.6) {$\bullet$};
	
	\node (3) at (2,2) {$\bullet$};
	\node (3a) at (1.6,1.6) {$\bullet$};
	\node (3b) at (1.8,1.6) {$\bullet$};
	\node (3c) at (2,1.6) {$\bullet$};	
	
	\node (a) at (0,0) {$\bullet$};
	\node (b) at (1,0) {$\bullet$};
	\node (c) at (2,0) {$\bullet$};

	\draw (1a) -- (1);
	\draw (1) -- (1b);
	\draw (1) -- (1c);

	\draw (2) -- (2a);
	\draw (2b) -- (2);
	\draw (2) -- (2c);

	\draw (3) -- (3a);
	\draw (3) -- (3b);
	\draw (3c) -- (3);

	\draw (1a) -- (a);
	\draw (b) -- (1b);
	\draw (c) -- (1c);

	\draw (a) -- (2a);
	\draw (2b) -- (b);
	\draw (c) -- (2c);

	\draw (a) -- (3a);
	\draw (b) -- (3b);
	\draw (3c) -- (c);
	
	\draw[draw=white,double=black, line width=0.8pt,-, double distance=.4pt] (1,.8)--(1,1.2);
	\draw[draw=white,double=black, line width=0.8pt,-, double distance=.4pt] (.4,1.2)--(.8,.4);
	\draw[draw=white,double=black, line width=0.8pt,-, double distance=.4pt] (1.6,1.2)--(1.2,.4);
	\draw[draw=white,double=black, line width=0.8pt,-, double distance=.4pt] (.7,1.4)--(.6,1.2);
	\draw[draw=white,double=black, line width=0.8pt,-, double distance=.4pt] (1.3,1.4)--(1.4,1.2);
\end{tikzpicture}
    \caption{The non-planar quiver used in \cref{ex:K33}.}
    \label{fig:K33}
\end{figure}

\begin{example}\label{ex:K33}
    Consider the path algebra $\Lambda=KQ/I$, where $Q$ is the quiver in \cref{fig:K33} and $I$ is any set of relations making $\Lambda$ gentle, see \cref{def:gentle}. Then $\Lambda$ has global dimension $2$.
\end{example}

\section{\texorpdfstring{$d$}{d}-cluster tilting subcategories ofthe  derived category
} \label{section: derived category of gentle algebra}

In this section we study the derived category of a gentle algebra. We discuss under what circumstances this category contains $d$-cluster tilting subcategories that are closed under $[d]$. Our main tool in this investigation is the geometric model from \cite{OPS}, as described in \cref{sec: geometric model}. All $\Hom$ sets in this section are considered in the derived category.

Our proofs are based on the observation below, which gives a useful condition satisfied by perfect objects contained in a $d$-cluster tilting subcategory closed under $[d]$. This result can be deduced from \cite{IY}*{Proposition 3.4} in the case where the algebra has finite global dimension. Note that \cref{prop: tau^-1 rigid} holds for any finite dimensional algebra.

\begin{proposition} \label{prop: tau^-1 rigid}
Let $\Lambda$ be a finite dimensional algebra and consider a $d$-cluster tilting subcategory $\U\subseteq\D^b(\Lambda)$ that is closed under $[d]$ for some $d \geq 2$. If an indecomposable perfect object $X$ is contained in $\U$, then $\Hom(\tau X,X) = 0$.
\end{proposition}

\begin{proof}
Recall that as the subcategory $\U$ is $d$-cluster tilting in $\D^b(\Lambda)$, 
we have
\begin{align*}
    \U
    = &\{Y\in \D^b(\Lambda) \mid \Hom(\U,Y[i])=0 ~\text{ for }~ 1\leq i \leq d-1\}\\
    = &\{Y\in \D^b(\Lambda) \mid \Hom(Y,\U[i])=0 ~\text{ for }~ 1\leq i \leq d-1\}.
\end{align*}
As in \cref{subsec: d-CT}, we denote the first of these sets by $\Ud$ and the second one by $\dU$.

Suppose towards a contradiction that there exists an indecomposable perfect object $X$ in $\U$ with $\Hom(\tau X,X)\neq 0$. This yields
\[
\Hom(\tau X[1],X[1]) \cong \Hom(\tau X,X)\neq 0.
\] 
Since $d \geq 2$ and $\U$ is $d$-cluster tilting, this means that $\tau X[1]\notin \dU = \Ud$. Thus, there exists an object $Y \in \U$ such that $\Hom(Y,\tau X[1][i])\neq 0$ for some $1\leq i\leq d-1$. 

Using that $\tau[1] \colon \D^b(\proj\Lambda) \rightarrow \D^b(\Lambda)$ is a partial Serre functor in the sense of \cite{OppermannPS}, see in particular \cite{OppermannPS}*{Example 3.7},  we obtain
\begin{align*}
    \Hom(Y,\tau X[1][i])
    & \cong \Hom(Y[-i],\tau X[1])\\
    & \cong D \Hom (X, Y[-i])\\
    & \cong D \Hom (X[d], Y[d-i]) \neq 0.
\end{align*}
Since $\U$ is closed under $[d]$, we have $X[d]\in\U$. Noting that $1\leq d-i\leq d-1$, this yields $Y\notin \Ud=\U$, which is a contradiction. 
\end{proof}

Our next aim is to show that for gentle algebras there are limitations on arcs corresponding to objects in a $d$-cluster tilting subcategory. Recall that we use the notation $(\gamma_X, f_X)$ for the graded arc that corresponds to a string object $X$ in the derived category and that a finite arc is called minimal if it is homotopy equivalent to a segment of the boundary with no marked points.

\begin{lemma} \label{lem:exclude-bridging-arcs} 
Let $\Lambda$ be a gentle algebra. If $X\in \D^b(\Lambda)$ corresponds to a finite graded arc that is not minimal, then $\Hom(\tau X, X)\neq 0$.
\end{lemma}

\begin{proof}
As the graded arc $\gamma_X$ corresponding to $X$ is finite, we can use the algorithm described in \cite{OPS}*{Section 5} to compute $\tau X$. We obtain the graded arc $\gamma_{\tau X}$ corresponding to $\tau X$ by moving the endpoints of $\gamma_X$ to the next marked points following the orientation of the boundary. This is illustrated in \cref{fig:tau_x}, where the marked points in the figure are not necessarily distinct.

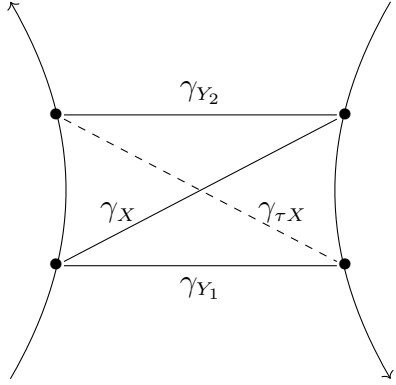
\begin{figure}[ht]
    \centering
\begin{tikzpicture}
\draw[->] (0,0) to[bend right] (0,5);
\draw[->] (5,5) to[bend right] (5,0);
\node[inner sep=0pt] (l1) at (.6,1.5) {$\bullet$};
\node[inner sep=0pt] (l2) at (.6,3.5) {$\bullet$};
\node[inner sep=0pt] (r1) at (4.4,1.5) {$\bullet$};
\node[inner sep=0pt] (r2) at (4.4,3.5) {$\bullet$};
\draw (l1) to node[above,pos=.2]{$\gamma_X$} (r2);
\draw (l1) to node[below]{$\gamma_{Y_1}$}(r1);
\draw (l2) to node[above]{$\gamma_{Y_2}$}(r2);
\draw[dashed] (l2) to node[above, pos=.8]{$\gamma_{\tau X}$} (r1);
\end{tikzpicture}    
    \caption{Geometric computation of $\tau X$.} 
    \label{fig:tau_x}
\end{figure}

Note that the arc $\gamma_X$ is minimal if and only if either $\gamma_{Y_1}$ or $\gamma_{Y_2}$ is contractible to a point. When $\gamma_X$ is not minimal, the objects $Y_1$ and $Y_2$ are hence non-zero. This yields an almost split sequence 
\[\tau X\rightarrow Y_1\oplus Y_2\rightarrow X\] 
with two indecomposable summands in the middle term. The morphism $\tau X\rightarrow Y_1\rightarrow X$ is non-zero, and the conclusion follows.
\end{proof}

As a consequence of our two previous results, we find that if an indecomposable perfect object is contained in a $d$-cluster tilting subcategory that is closed under $[d]$, then it corresponds to a minimal graded arc. 

\begin{lemma}\label{lem:all minimal}
Let $\Lambda$ be a gentle algebra and consider a $d$-cluster tilting subcategory $\U \subseteq \D^b(\Lambda)$ that is closed under $[d]$ for some $d \geq 2$. If an indecomposable perfect object is contained in $\U$, then it corresponds to a minimal graded arc.
\end{lemma}

\begin{proof}
Let $X$ be a perfect object that is contained in $\U$. If $X$ is a band object, then $\tau X\cong X$, which contradicts \cref{prop: tau^-1 rigid}. As $X$ is perfect, this implies that $\gamma_X$ is a finite graded arc. Our statement now follows by combining \cref{lem:exclude-bridging-arcs} and \cref{prop: tau^-1 rigid}.
\end{proof}

We next investigate how the gradings of minimal arcs giving rise to objects in a $d$-cluster tilting subcategory are related. Let $\Lambda$ be a gentle algebra and recall that $L$ denotes the associated dual graph as explained in \cref{sec: geometric model}.

Consider two graded arcs $(\gamma_X, f_X)$ and $(\gamma_Y, f_Y)$ with a common endpoint $m$, as indicated in \cref{fig:minimal help}. Let $p_X$ (resp.\ $p_Y$) denote the intersection point of $\gamma_X$ (resp.\ $\gamma_Y$) and $L$ that is closest to $m$. We say that $\gamma_X$ and $\gamma_Y$ have \emph{compatible grading in $m$} if $f_X(p_X) = f_Y(p_Y)$. Note that by the description of morphisms in the bounded derived category in terms of graded intersections, this implies that $\Hom(Y,X) \neq 0$. Similarly, we say that the two arcs have \emph{$d$-compatible grading in $m$} if $f_X(p_X) \equiv f_Y(p_Y) \pmod{d}$. 

\begin{figure}[hbt]
    \centering
    \begin{tikzpicture}
        \draw (-1,0) -- (9,0);
        \filldraw[red] (2,0) circle[radius=.1];
        \filldraw[red] (6,0) circle[radius=.1, fill=red];
        \draw[red] (2,0) -- node[right, pos=0.35]{$p_X$} (3,2);
        \draw[red] (6,0) --node[left, pos=0.35]{$p_Y$} (5.5,2);
        \filldraw (4,0) circle[radius=.1, fill=black];
        \draw (4,0) to node[above, pos=.8] {$\gamma_X$} (1,1);
        \draw[dashed] (1,1) -- (-.5, 1.5);
        \draw (4,0) to node[above, pos=.8] {$\gamma_Y$}(7,1);
        \draw[dashed] (7,1) -- (8.5, 1.5);
        \node[below] at (4,0) {$m$};
    \end{tikzpicture}
    \caption{Illustration of compatible grading. The graded arcs $\gamma_X$ and $\gamma_Y$ have $d$-compatible grading in $m$ if $f_X(p_X) \equiv f_Y(p_Y) \pmod{d}$.}
    \label{fig:minimal help}
\end{figure}
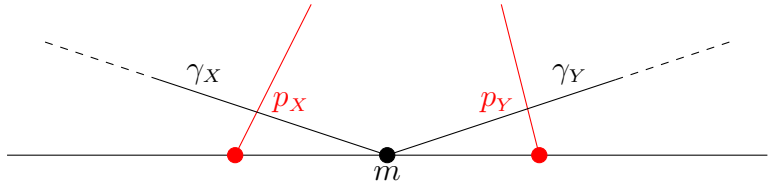

Consider an object $X$ in $\D^b(\Lambda)$ corresponding to a minimal graded arc. Let $\B_X$ be the boundary component containing the endpoints of $\gamma_X$. Denote the marked points on $\B_X$ by $m_1, m_2, \dots, m_t$, where $m_1$ and $m_t$ are the endpoints of $\gamma_X$. The marked point $m_{j+1}$ follows $m_j$, as illustrated in \cref{fig:naming minimal}. Use the notation $\gamma_1 = \gamma_X$ for the graded arc corresponding to $X_1 = X$. For $j = 2,\dots,t$, let $\gamma_j$ be the minimal graded arc with endpoints $m_{j-1}$ and $m_{j}$ for which $\gamma_{j-1}$ and $\gamma_j$ have compatible grading in $m_{j-1}$. Denote the object corresponding to $\gamma_j$ by $X_j$. We now define $\V_X^d$ to be the subcategory
\[
\V_X^d = \add\{X_j[di] \mid i \in \Z ~\text{  and  }~ 1 \leq j \leq t\} \subseteq \D^b(\Lambda).
\]

Note that as every indecomposable object in $\V_X^d$ corresponds to a minimal graded arc, we have $\V_X^d \subseteq \D^b(\proj \Lambda)$.

\begin{figure}[htb]
\begin{tikzpicture}[radius=5cm, scale=.6]
	\coordinate (dotted_start) at (350:5);
	\draw (dotted_start) arc[start angle= 350, delta angle= 270] node[pos=.17, above right] {$\B_X$};
	\draw[dashed] (dotted_start) arc[start angle=350, delta angle=-90];
	\coordinate (m3) at (220:5);
	\coordinate (m2) at (165:5);
	\coordinate (m1) at (120:5);
	\coordinate (mt) at (60:5);
	\coordinate (m-) at (10:5);
	\filldraw (m3) circle[radius=.1] node[left]{$m_3$};
	\filldraw (m2) circle[radius=.1] node[left]{$m_2$};
	\filldraw (m1) circle[radius=.1] node[above]{$m_1$};
	\filldraw (mt) circle[radius=.1] node[right]{$m_t$};
	\filldraw (m-) circle[radius=.1] node[right]{$m_{t-1}$};
	\draw (m3) to[bend right] node[right]{$\gamma_3$} (m2);
	\draw (m2) to[bend right] node[right]{$\gamma_2$} (m1);
	\draw (m1) to[bend right] node[below]{$\gamma_1=\gamma_X$} (mt);
	\draw (mt) to[bend right] node[left]{$\gamma_t$} (m-);
	\draw[dashed] (m3) to[bend left] (0,-4);
	\draw[dashed] (m-) to[bend right] (3,-3);
\end{tikzpicture}
    \caption{Minimal arcs sharing a boundary with $\gamma_X$.}
    \label{fig:naming minimal}
\end{figure}
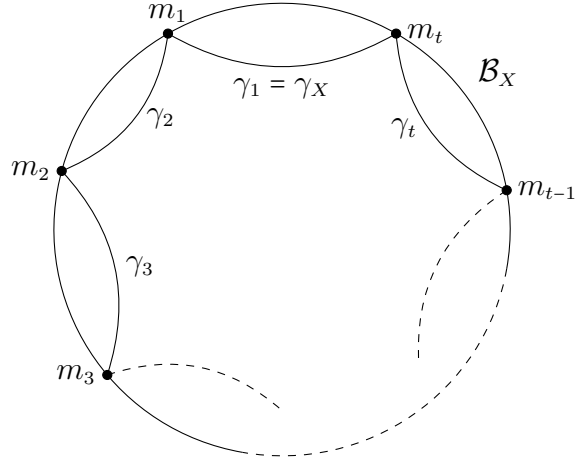

\begin{example}\label{ex:V_X^d}
In \cref{ex:d-CT} we considered the algebra $\Lambda\! =\! KA_3/J^2$, which is $2$-representation finite $2$-hereditary. Note that the $2$-cluster tilting subcategory $\U[2\Z] \subseteq \D^b(\Lambda)$ as illustrated in \cref{fig:derA3} can also be obtained as $\V_{X}^2$ for any indecomposable object $X\in\U[2\Z]$.
\end{example}

Our next two lemmas demonstrate that the notion of $d$-compatibility is useful to describe when indecomposable objects are contained in a $d$-cluster tilting subcategory. In \cref{lem:d-compatibility general} we consider the general case, before restricting to perfect objects in \cref{lem:only minimal}.

\begin{lemma}\label{lem:d-compatibility general}
Let $\Lambda$ be a gentle algebra and consider a $d$-cluster tilting subcategory $\U \subseteq \D^b(\Lambda)$ that is closed under $[d]$ for some $d \geq 2$. The following statements hold for an indecomposable object $X$ in $\U$: 
\begin{enumerate}
    \item $X[j]\in \U$ if and only if $d$ divides $j$.
    \item Let $Y$ correspond to a graded arc $\gamma_Y$ and assume that $\gamma_X$ and $\gamma_Y$ have a common endpoint $m$. If $Y \in \U$, then $\gamma_X$ and $\gamma_Y$ have $d$-compatible grading in $m$.
\end{enumerate}
\end{lemma}

\begin{proof}
Notice that as the subcategory $\U$ is closed under $[d]$, it is also closed under $[id]$ for every integer $i$. It follows that if $d$ divides $j$, then $X[j]\in \U$. Suppose that $d$ does not divide $j$. We may assume without loss of generality that $1 \leq j \leq d-1$. As we have $\Hom(X[j], X[j])\neq 0$, this implies that $X[j]$ is not contained in $\dU=\U$, which proves the first statement.

For the second statement, denote by $p_X$ (resp.\ $p_Y$) the intersection point of $\gamma_X$ (resp.\ $\gamma_Y$) and $L$ that is closest to $m$. Assume that $\gamma_X$ and $\gamma_Y$ do not have $d$-compatible grading in $m$, i.e.\ that $~f_X(p_X) \nequiv f_Y(p_Y) \pmod{d}$. By \textit{(1)}, it is enough to consider the case where $f_X(p_X)=0$ and $f_Y(p_Y)=i$ for some $1\leq i\leq d-1$. It follows from the description of morphisms in terms of graded intersections that this yields $\Hom(X,Y[i])\neq 0$ or $\Hom(Y[i],X)\neq 0$. In either case, this implies that $Y$ is not contained in $\U$ by the definition of a $d$-cluster tilting subcategory.
\end{proof}

For perfect objects, there is a converse to part \textit{(2)} of the lemma above, as described in \cref{lem:only minimal} part \textit{(1)}. This yields a complete description of the perfect objects that are contained in a $d$-cluster tilting subcategory $\U$ that is closed under $[d]$. In particular, we see that the inclusion of one indecomposable perfect object $X$ in $\U$ fully determines what other perfect objects associated to the boundary $\B_X$ are contained in $\U$. Note that the object $X$ in the lemma below must correspond to a minimal graded arc by \cref{lem:all minimal}.

\begin{lemma}\label{lem:only minimal}
Let $\Lambda$ be a gentle algebra and consider a $d$-cluster tilting subcategory $\U \subseteq \D^b(\Lambda)$ that is closed under $[d]$ for some $d \geq 2$. The following statements hold for an indecomposable perfect object $X$ in $\U$:
\begin{enumerate}
    \item Let $Y$ correspond to a minimal graded arc $\gamma_Y$ that shares an endpoint $m$ with $\gamma_X$. Then $Y \in \U$ if and only if $\gamma_X$ and $\gamma_Y$ have $d$-compatible grading in $m$. 
    \item Let $Y$ correspond to a minimal graded arc with endpoints on the same boundary component as $\gamma_X$. Then $Y \in \U$ if and only if $Y \in \V_X^d$. In particular, we have $\V_X^d \subseteq \U$.
\end{enumerate}
\end{lemma}

\begin{proof}
If $Y\in\U$, the arcs $\gamma_X$ and $\gamma_Y$ have $d$-compatible grading in $m$ by part \textit{(2)} of  \cref{lem:d-compatibility general}. 
For the reverse direction, assume that $Y \notin \U$. As $\U=\dU$, this means that there is an indecomposable object $Z \in \U$ such that $\Hom(Y,Z[i]) \neq 0$ for some $1 \leq i \leq d-1$. Without loss of generality we can assume the intersection of $\gamma_X$ and $\gamma_Y$ in $m$ to be of the type illustrated in \cref{fig:minimal help}, as the proof in the case where the position of the two arcs is interchanged is dual. Consequently, using that $\gamma_Y$ is minimal, a non-zero morphism from $Y$ to $Z[i]$ must correspond to a graded intersection of $\gamma_{Y}$ and $\gamma_{Z[i]}$ in the endpoint $m$. In particular, this implies that the arcs $\gamma_{Y}$ and $\gamma_{Z}$ do not have $d$-compatible grading in $m$. On the other hand, the gradings of $\gamma_X$ and $\gamma_Z$ are $d$-compatible in $m$ by \cref{lem:d-compatibility general} part \textit{(2)}, as $X$ and $Z$ are in $\U$. Combining this, we see that $\gamma_X$ and $\gamma_Y$ do not have $d$-compatible grading in $m$.

The second statement  follows directly.
\end{proof}

Using the description in \cref{lem:only minimal}, we are now able to prove that the only examples of $d$-cluster tilting subcategories of the perfect derived category that are closed under $[d]$ arise in the type $A$ case. This result provides an important step in the proof of \cref{thm:main result}, where we show that the analogue statement holds for $\D^b(\Lambda)$.

\begin{remark}
Note that \cref{prop: tau^-1 rigid}, \cref{lem:all minimal}, \cref{lem:d-compatibility general} and  \cref{lem:only minimal} hold also when considering a $d$-cluster tilting subcategory $\U \subseteq \D^b(\proj\Lambda)$. These results can hence be applied in the setup of \cref{thm:nothing in finite glob dim}.
\end{remark}

\begin{proposition}\label{thm:nothing in finite glob dim}
Let $\Lambda$ be a gentle algebra. If $\D^b(\proj\Lambda)$ contains a $d$-cluster tilting subcategory that is closed under $[d]$ for some $d\geq 2$, then $\Lambda$ is derived equivalent to an algebra of Dynkin type $A$.
\end{proposition}

\begin{proof}
Assume that $\Lambda$ is not derived equivalent to an algebra of Dynkin type $A$. By \cite{OPS}*{Corollary 1.23}, this means that the surface in the geometric model of $\D^b(\Lambda)$ is not a disk.

Assume towards a contradiction that there exists a $d$-cluster tilting subcategory $\U$ of $\D^b(\proj\Lambda)$ that is closed under $[d]$ for some $d \geq 2$. Let $X$ be an indecomposable object in $\U$, which by \cref{lem:all minimal} corresponds to a minimal graded arc $(\gamma_X,f_X)$ as $X$ is perfect. The boundary component containing the endpoints of $\gamma_X$ is denoted by $\B_X$. \cref{lem:only minimal} part \textit{(2)} yields that an object corresponding to a minimal graded arc with endpoints on $\B_X$ is contained in $\U$ if and only if it is in $\V_X^d$.

Denote one of the endpoints of $\gamma_X$ by $m$. As the surface in the geometric model associated to $\Lambda$ is not a disk, there exists an arc $\gamma$ starting and ending in $m$ that is not contractible to a point.  For simplicity, we assume that $\gamma$ has no self-intersections except in the endpoint. Let $\gamma_i$ be the arc obtained by concatenating $i$ copies of $\gamma$ and $\gamma_X$ in such a way that $\gamma_i$ has $i$ self-intersections. Note that when choosing a representative of $\gamma_i$, it is helpful to choose one such that all the self-intersections are separated and occur before the concatenation with $\gamma_X$. See \cref{fig:bandlike_object} for an illustration of $\gamma_2$.

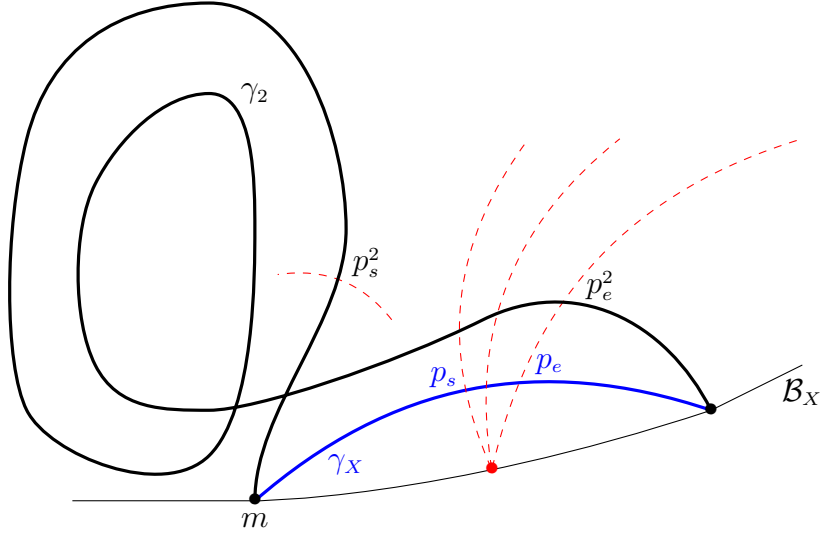
\begin{figure}[hbt]
    \centering
    \begin{tikzpicture}[inner sep=0, scale=.6, radius=8cm]
    \coordinate (m) at (0,-8);
    \coordinate (m2) at (10,-6);
    \draw (-4,-8)--(m) .. controls(0,-8) and (4,-8).. node[pos=.7] (l) {}  (m2)--(12, -5);
    \begin{scope}[petal]
    \draw (m) to[bend left] (m2);
    \end{scope}
    \draw[laminate, dashed] (l) to[bend left] node[pos=0.22, right, inner sep=3,color=blue] {$p_e$} node[pos=0.42, right, inner sep=4,color=black] {$p_e^2$} (12,0);
    \draw[laminate, dashed] (l) to[bend left] (8,0);
    \draw[laminate, dashed] (l) to[bend left] node[pos=0.26, left, inner sep=2,color=blue] {$p_s$}(6,0);
    \draw[laminate, dashed] (3,-4) to[bend right] node[pos=0.45, above right,color=black, inner sep=2] {$p_s^2$} (.5,-3);
    \begin{scope}[arc]
    \draw (m) ..controls(0,-6) and (2,-4).. 
                (2,-2) ..controls (2,0) and(1,3)..
                (-1,3) ..controls (-3,3) and(-4.5,2)..
                (-5,0) ..controls (-5.5,-2) and(-5.5,-5)..
                (-5,-6) ..controls (-4.5,-7) and(-2, -8)..
                (-1,-7) ..controls (0,-6) and(0,-3)..
                (0,-2) ..controls (-0,-1) and(0,1)..
                (-1,1) ..controls (-2,1) and(-3,0)..
                (-3.5,-1) ..controls (-4,-2) and(-4,-4)..
                (-3.5,-5) ..controls (-3,-6) and(-2,-6)..
                (-1,-6) ..controls(0,-6) and (3,-5)..
                (5,-4) ..controls(7,-3) and (9,-4).. 
                (m2);
    \node[petal] at (2,-7.2) {$\gamma_X$};
    \node[arc] at (0,1) {$\gamma_2$};
    \node[arc] at (12,-5.6) {$\B_X$};
    \node (naked) at (m) {$\bullet$};
    \node (next) at (m2) {$\bullet$};
    \node[laminate] (red) at (l) {$\bullet$};
    \node[anchor=north, inner sep=4] at (m) {$m$};     
    \end{scope}
    \end{tikzpicture}
    \caption{The arc $\gamma_i$ for $i=2$, as used in the proof of \cref{thm:nothing in finite glob dim}.}
    \label{fig:bandlike_object}
\end{figure}

Denote by $p_s$ (resp.\ $p_s^i$) the intersection point of $\gamma_X$ (resp.\ $\gamma_i$) and $L$ that is closest to $m$, as indicated in \cref{fig:bandlike_object}. Similarly, we use the notation $p_e$ (resp.\ $p_e^i$) for the intersection point of $\gamma_X$ (resp.\ $\gamma_i$) and $L$ that is closest to the second endpoint of $\gamma_X$. Note that $p_s$ and $p_e$ might coincide. Let $f_i$ be the grading of $\gamma_i$ for which $f_i(p^i_s)=f_X(p_s)$. Denote by $Y_i$ the indecomposable object in $\D^b(\proj\Lambda)$ corresponding to the finite graded arc $(\gamma_i,f_i)$.

Notice that for each loop of $\gamma_i$, the grading increases by an integer $w$. As the last segment in the construction of $\gamma_i$ follows the trajectory of $\gamma_X$, this yields $f_i(p_e^i)=f_X(p_e)+iw$. Choosing $i=d$, we hence obtain $f_d(p_e^d)\equiv f_X(p_e)\pmod d$. By construction, the arc $\gamma_d$ does not intersect any minimal arcs except in the endpoints. This implies that
\[
\Hom(\U,Y_d[j]) = \Hom(\V_X^d,Y_d[j]) = 0
\]
for $1 \leq j \leq d-1$.
Consequently, the object $Y_d$ is contained in $\Ud = \U$. This contradicts \cref{lem:all minimal} as $\gamma_d$ is not a minimal arc.
\end{proof}

Our next aim is to show that the same conclusion as in \cref{thm:nothing in finite glob dim} holds when working with the entire bounded derived category. We hence need to consider objects corresponding to graded arcs wrapping around punctures on one or both ends. To determine whether such objects can be contained in a $d$-cluster tilting subcategory, the following lemma is useful.

\begin{lemma}\label{lem:no crossings}
Let $\Lambda$ be a gentle algebra and consider a $d$-cluster tilting subcategory $\U \subseteq \D^b(\Lambda)$ that is closed under $[d]$ for some $d \geq 2$. If $X$ and $Y$ are two indecomposable objects in $\U$, then the corresponding graded arcs $\gamma_X$ and $\gamma_Y$ do not intersect in the interior of the surface in the geometric model.
\end{lemma}

\begin{proof}
Let $X$ and $Y$ be in $\U$ and assume that $\gamma_X$ and $\gamma_Y$ do intersect in the interior of the surface in the geometric model. This intersection lies in a polygon $P$ of the dissection given by the dual graph $L$ as illustrated in \cref{fig:graded internal intersection}. Denote intersection points of $\gamma_X$ and $\gamma_Y$ with the edges of $P$ as indicated in the figure. Using that $\Hom(Y,X[i])=0$ for $1 \leq i \leq d-1$ and applying \cref{lem:d-compatibility general} part \textit{(1)}, we can assume $f_X(p_X)=f_Y(p_Y)$. As $f_Y(p'_Y)=f_Y(p_Y)+1$, this implies that $\Hom(X,Y[1])\neq 0$, which yields a contradiction.
\end{proof}

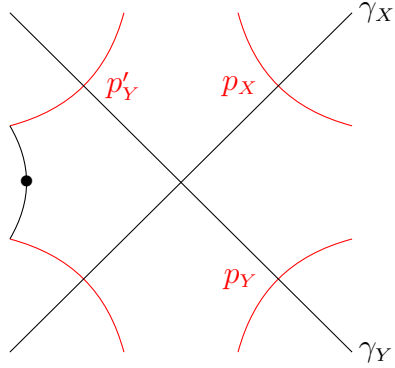
\begin{figure}[thb]
    \centering
    
    \pgfmathsetmacro{\dist}{3}
    \pgfmathsetmacro{\nodedisp}{5pt}
    \begin{tikzpicture}[scale=1.5, circle, inner sep=1pt]
	    \coordinate (bottomleft) at (0,0);
	    \coordinate (topleft) at (0,\dist);
	    \coordinate (bottomright) at (\dist,0);
	    \coordinate (topright) at (\dist,\dist);
	
	    \path[laminate] (bottomleft) edge[white] +(0,1) +(0,1)  edge [laminate, bend left] node[right=\nodedisp] {} +(1,0) coordinate (bottomlefttop);
	    \path[laminate] (topleft) edge[white] +(0,-1) +(0,-1)  edge [laminate, bend right] node[right=\nodedisp] {$p_Y'$} +(1,0) coordinate (topleftbottom);
	    \path[laminate] (bottomright) edge[white] +(0,1)	+(0,1) edge [laminate, bend right] node[left=\nodedisp] {$p_Y$} +(-1,0) coordinate (bottomrighttop);
	    \path[laminate]	(topright) edge[white] +(0,-1) +(0,-1)  edge [laminate, bend left] node[left=\nodedisp] {$p_X$} +(-1,0) coordinate (toprightbottom);
	    \draw[bend right] (bottomlefttop) to node {$\bullet$} (topleftbottom);
	
	    \draw (bottomleft)-- node[right=\nodedisp] {} node[at end, right] {$\gamma_X$} (topright);
	    \draw (topleft)-- node[at end, right] {$\gamma_Y$} (bottomright);
	\end{tikzpicture}
	
    \caption{The oriented graded intersection discussed in the proof of \cref{lem:no crossings}.}
    \label{fig:graded internal intersection}
\end{figure}

We are now ready to prove \cref{thm:main result}.

\begin{theorem}\label{thm:main result}
Let $\Lambda$ be a gentle algebra. If $\D^b(\Lambda)$ contains a $d$-cluster tilting subcategory that is closed under $[d]$ for some $d\geq 2$, then $\Lambda$ is derived equivalent to an algebra of Dynkin type $A$.
\end{theorem}

\begin{proof}

If $\Lambda$ has finite global dimension, then the bounded derived category coincides with the perfect derived category and the result follows from \cref{thm:nothing in finite glob dim}. We hence assume that $\Lambda$ has infinite global dimension, which is equivalent to the existence of at least one puncture in the surface of the associated geometric model. The strategy from here is to show that if there exists a $d$-cluster tilting subcategory $\U \subseteq \D^b(\Lambda)$ that is closed under $[d]$, then $\U$ contains a certain non-perfect object $X$. This enables us to construct a perfect object $Z$  that must be contained in $\U$ but does not correspond to a minimal graded arc. Similarly as in the proof of \cref{thm:nothing in finite glob dim}, this leads to a contradiction.

So suppose there exists a $d$-cluster tilting subcategory $\U\subseteq\D^b(\Lambda)$ that is closed under $[d]$. We claim that $\U$ must contain an object $X$ corresponding to a graded arc that starts in a marked point and ends wrapping around a puncture. To see this, assume to the contrary that each indecomposable object in $\U$ is either perfect or corresponds to a graded arc wrapping around punctures on both ends. Let $\B$ denote a boundary component with at least one marked point, and consider the arc $\gamma$ that starts and ends in this marked point and follows the boundary $\B$ up to homotopy. Note that by our assumptions, this arc is not contractible to a point. Applying the iterative construction from the proof of \cref{thm:nothing in finite glob dim} to $\gamma$, we obtain a finite arc $\gamma_d$. This arc can be equipped with a grading that is $d$-compatible with the grading of any minimal graded arc that corresponds to an object in $\U$ and starts or ends in the endpoints of $\gamma_d$. By \cref{lem:all minimal} combined with our assumption on the non-perfect objects in $\U$, this yields that the object $Y_d$ corresponding to the graded arc $\gamma_d$ is contained in $\U$. Similarly as in the proof of \cref{thm:nothing in finite glob dim}, this is a contradiction as $\gamma_d$ is not minimal.

Thus, we can assume that $\U$ contains an indecomposable object $X$ such that the corresponding graded arc $(\gamma_X,f_X)$ starts in a marked point $m$ and ends wrapping around a puncture $r$. Before defining the object $Z$, we label some useful points in the model and make some observations. Denote by $l$ an edge in the dual graph $L$ that is adjacent to $r$. Let $p$ be an intersection of $\gamma_X$ and $l$ such that after this intersection, the arc $\gamma_X$ wraps infinitely many times around the single puncture $r$. The next intersection of $\gamma_X$ and $l$ is denoted by $p'$, as indicated in \cref{fig:infinite bandlike object}. Note that by our assumption on $p$, any arc that intersects $l$ between $p$ and $p'$ and does not intersect $\gamma_X$, has an end that wraps infinitely many times around $r$. Let $w$ be the integer defined by the equation $f_X(p')=f_X(p)+w$, and observe that the grading at the intersections of $\gamma_X$ and $l$ increases by $w$ each time $\gamma_X$ loops around $r$. By the description of morphisms arising from punctures, we deduce that $d$ divides $w$ as $X$ is in $\U$. 

We now construct the arc $\gamma_Z$ by concatenating the following four segments:
\begin{enumerate}[label=(\roman*)]
    \item The first segment of $\gamma_Z$ starts in $m$ and follows the trajectory of $\gamma_X$ until the point $p'$;
    \item The second segment of $\gamma_Z$ follows $l$ from $p'$ to $p$;
    \item The third segment of $\gamma_Z$ follows the trajectory of $\gamma_X$ from $p$ and back to $m$;
    \item The last segment of $\gamma_Z$ follows the minimal arc $\gamma_m$ starting in $m$ and ending in a marked point $m'$ in such a way that $\gamma_Z$ has exactly one self-intersection.
\end{enumerate}
We equip $\gamma_Z$ (resp.\ $\gamma_m$) with the grading $f_Z$ (resp.\ $f_m$) that is compatible with $f_X$ in $m$ and denote the indecomposable object in $\D^b(\Lambda)$ corresponding to $(\gamma_Z,f_Z)$ by $Z$.

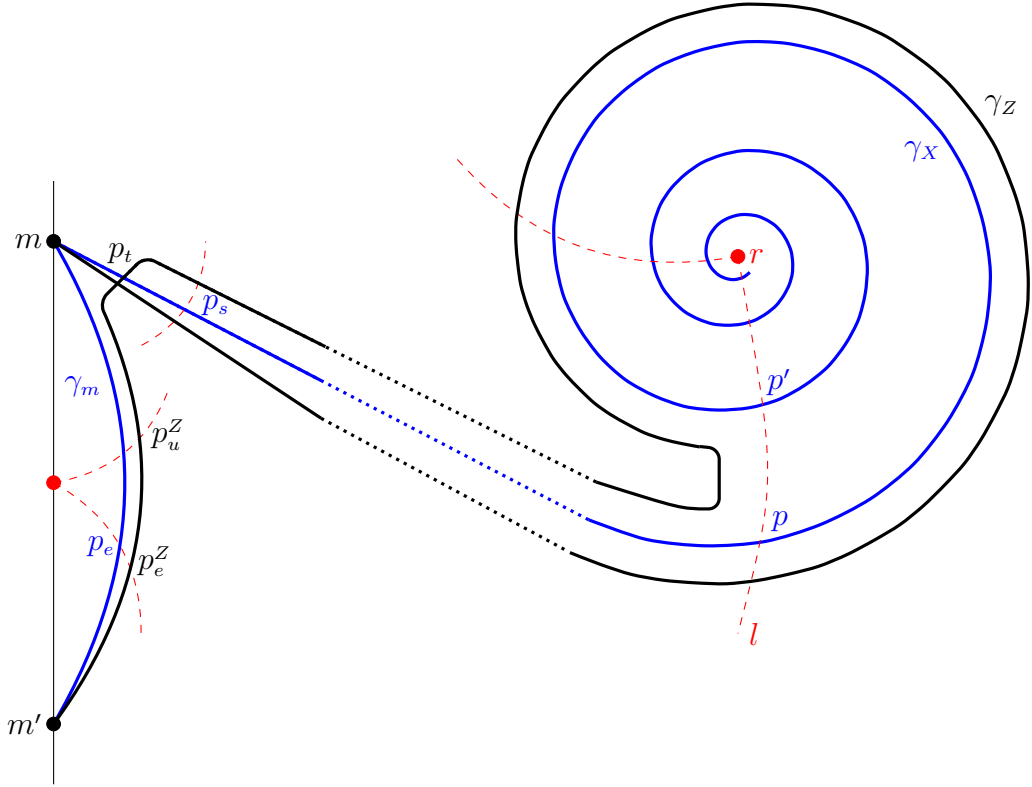
\begin{figure}[t!hb]
    \centering
    \begin{tikzpicture}
	\draw (0,0) -- coordinate[pos=0.1](m2) coordinate[pos=0.1](m2)coordinate[pos=0.9](m) coordinate[pos=0.5](r2) (0,8);
    \coordinate (r) at (9,7); 
    \begin{scope}[petal]
    	\draw (m) to[bend left] node[pos=0.3, left]{$\gamma_m$} (m2);
    	\begin{scope}[shift=(r)] 
	    	\draw [domain=13.5:33.5,variable=\t,smooth,samples=100] plot ({-\t r}: {0.0001*\t*\t*\t}) coordinate(endspiral); 
	    \end{scope}
	    \draw[dotted] (endspiral)-- coordinate[pos=0.5] (midpoint) (m);
	    \draw (midpoint) -- (m);
    \end{scope}
    \begin{scope}[laminate,dashed]
    	\draw (r) .. controls +(280:3) .. node[at end,right]{$l$}+(270:5);		
    	\draw[bend left] (r) to +(160:4);		
    	\draw (r2) to[bend right] +(40:2);   
    	\draw (r2) to[bend left] +(-60:2.3);  
    	\draw ([shift={(-50:1.8)}]m) to[bend right] ([shift={(0:2)}]m);  
    \end{scope}
    \draw[very thick] (m) -- ([shift={(0,-.5)}]midpoint);
    \draw[dotted, very thick] ([shift={(0,-.5)}]midpoint) -- ([shift={(236:.5)}]endspiral);
    \begin{scope}[shift=(r), very thick] 
	    	\draw [domain=26.9:33.5,variable=\t,smooth,samples=70] plot ({-\t r}: {0.0001*\t*\t*\t+.5}); 
	    	\draw[rounded corners] ({-26.9 r}:2.57) -- ({-26.8 r}:2.56) -- 
	    		+(0,-.8)  .. controls +(-.5,0).. ([shift={(.1,.5)}]endspiral); 
	    	\draw[dotted, rounded corners] ([shift={(.1,.5)}]endspiral) -- coordinate[pos=0.6] (midpointblack) ([shift={(1.2,-.2)}]m) -- ++(-.3,-.3); 
	    	\draw[rounded corners]  (midpointblack) --([shift={(1.2,-.2)}]m)-- ++(-.6,-.6) to[bend left]
	    	coordinate[pos=0.3](firstcross) coordinate[pos=0.6](lastcross) (m2); 
	    \end{scope}
	    
    \filldraw (m) circle (.09cm) node[left] {$m$};
    \filldraw (m2) circle (.09cm) node[left] {$m'$};
    \filldraw[laminate] (r) circle (.09cm) node[right] {$r$};
    \filldraw[laminate] (r2) circle (.09cm);
    \node[petal] at ([shift={(338.3:2.3)}]m) {$p_s$};
    \node at ([shift={(350:.9)}]m) {$p_t$};
    \node at ([shift={(.3,0)}]lastcross) {$p_e^{Z}$};
    \node[petal] at ([shift={(-.4,.2)}]lastcross) {$p_e$};
    \node at ([shift={(.4,0)}]firstcross) {$p_u^{Z}$};
    \node[petal] at ([shift={(279:3.55)}]r) {$p$};
    \node[petal] at ([shift={(288:1.75)}]r) {$p'$};
    \node at ([shift={(30:4)}]r) {$\gamma_{Z}$};
    \node[petal] at ([shift={(30:2.8)}]r) {$\gamma_X$};
\end{tikzpicture}

    \caption{The arc $\gamma_Z$ as used in the proof of \cref{thm:main result}}
    \label{fig:infinite bandlike object}
\end{figure}

The last step in our proof is to show that $Z \in \dU = \U$.
Let $p_e$ (resp.\ $p_e^Z$) denote the intersection point of $\gamma_m$ (resp.\ $\gamma_Z$) that is closest to the marked point $m'$. By similar arguments as in the proof of \cref{thm:nothing in finite glob dim}, we see that $f_Z(p_e^Z)=f_m(p_e)+w$. As $d$ divides $w$, this implies that the gradings of $\gamma_Z$ and $\gamma_m$ are $d$-compatible in $m'$. \cref{lem:d-compatibility general} part \textit{(2)} hence allows us to conclude that intersections in the endpoints of $\gamma_Z$ give rise to no non-zero morphisms from $Z$ to $\U[i]$ for $1\leq i \leq d-1$.

We claim that the same holds for any intersection between $\gamma_Z$ and a graded arc corresponding to an object in $\U$ in the interior of the surface. Notice first that this is immediate for the segments of $\gamma_Z$ described in (i) and (iii), as an intersection here would also yield an intersection with $\gamma_X$ and hence contradict \cref{lem:no crossings}. If a graded arc corresponding to an object in $\U$ intersects $\gamma_Z$ in the interior along the segment described in (iv), this yields a graded intersection of the same type as the  one denoted by $p_t$ in \cref{fig:infinite bandlike object}.

To study this graded intersection, let $p_s$ (resp.\ $p_u^Z$) denote the intersection of $\gamma_X$ (resp.\ the fourth segment of $\gamma_Z$) and $L$ that is closest to $m$, as indicated in \cref{fig:infinite bandlike object}. By similar arguments as before, we have $f_{Z}(p_u^{Z})=f_X(p_s)+w$. Again using that $d$ divides $w$, this implies that the intersection in $p_t$ does not give rise to non-zero morphisms from $Z$ to $\U[i]$ for $1\leq i \leq d-1$.

It remains to consider intersections of graded arcs corresponding to objects in $\U$ along the segment of $\gamma_{Z}$ described in (ii). As such a graded arc does not intersect $\gamma_X$, it wraps around $r$ on one end by the assumption on $p$. Close to $r$, the grading of such an arc hence agrees with that of $\gamma_X$ up to shifts by $d$. Using this, we look at the polygon of the dissection given by $L$ that the intersection lies in. Considering different possibilities for the marked point of this polygon, we see that an intersection of this type also does not give rise to non-zero morphisms from $Z$ to $\U[i]$ for $1\leq i \leq d-1$. This allows us to conclude that $Z \in \dU = \U$, which contradicts \cref{lem:all minimal} as $\gamma_{Z}$ is finite but not minimal.
\end{proof}

The fact that a gentle algebra is derived equivalent to an algebra of Dynkin type $A$ if and only if the surface in the associated geometric model is a disk \cite{OPS}*{Corollary 1.23}, played an important role in the proofs of \cref{thm:nothing in finite glob dim} and \cref{thm:main result}. We now move on to characterizing $d$-cluster tilting subcategories of the derived category in this case. Our classification gives a geometric interpretation of the $d$-cluster tilting subcategories of the derived category arising from $d$-representation finite $d$-hereditary gentle algebras, as described in \cref{subsec: d-CT}. By our classification in \cref{cor: d-RF gentle}, these algebras are of the form $KA_n/J^2$ for $n=d+1$.

\begin{theorem} \label{prop:type A}
Assume $n \geq 3$ and let $\Lambda$ be a gentle algebra that is derived equivalent to an algebra of Dynkin type $A_n$. A subcategory $\U \subseteq \D^b(\Lambda)$ is $d$-cluster tilting and closed under $[d]$ for some $d\geq 2$ if and only if $d=n-1$ and $\U = \V_X^d$ for some object $X$ corresponding to a minimal graded arc.
\end{theorem}

\begin{proof}

Without loss of generality, we can assume $\Lambda \cong KA_n/J^2$. As demonstrated in \cref{ex:quiver from surface}, the geometric model associated to $\Lambda$ is a disk with $n+1$ marked points on the boundary. The dual graph $L$ is shown in \cref{ex:dual graph}.

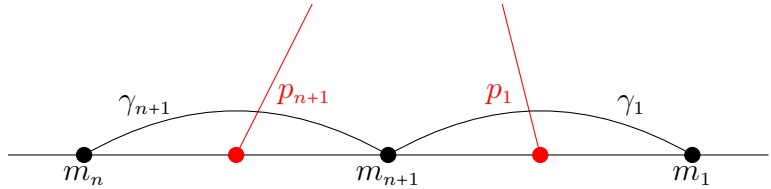
\begin{figure}[bht]
    \centering
    \begin{tikzpicture}
        \draw (-1,0) -- (9,0);
        \filldraw[red] (2,0) circle[radius=.1];
        \filldraw[red] (6,0) circle[radius=.1, fill=red];
        \draw[red] (2,0) -- node[right, pos=0.4]{$p_{n+1}$} (3,2);
        \draw[red] (6,0) --node[left, pos=0.4]{$p_{1}$} (5.5,2);
        \filldraw (0,0) circle[radius=.1, fill=black];
        \filldraw (4,0) circle[radius=.1, fill=black];
        \filldraw (8,0) circle[radius=.1, fill=black];
        \draw (0,0) to[bend left] node[above, pos=.2] {$\gamma_{n+1}$} (4,0);
        \draw (4,0) to[bend left] node[above, pos=.8] {$\gamma_1$}(8,0);
        \node[below] at (4,0) {$m_{n+1}$};
        \node[below] at (0,0) {$m_{n}$};
        \node[below] at (8,0) {$m_{1}$};
    \end{tikzpicture}
    \caption{The notation used in the proof of \cref{prop:type A}}
    \label{fig:lemma type A}
\end{figure}

Given an object $X$ in $\D^b(\Lambda)$ corresponding to a minimal graded arc $\gamma_X$, we follow the notation introduced in the definition of $\V_X^d$ with $t=n+1$, see \cref{fig:naming minimal}. We use the notation $f_i$ for the grading of $\gamma_i$. Consider the intersection $p_1$ (resp.\ $p_{n+1}$) of $\gamma_1 = \gamma_X$ (resp.\ $\gamma_{n+1}$) with the dual graph $L$ that is closest to the endpoint $m_{n+1}$, as indicated in \cref{fig:lemma type A}. By the description of $L$ and the compatibility of grading from the construction of $\V_X^d$, we see that
\begin{equation}\tag{$\ast$}
    f_{n+1}(p_{n+1}) = f_1(p_1)+n-1.
\end{equation}

Assume that $\U \subseteq \D^b(\Lambda)$ is $d$-cluster tilting and closed under $[d]$ for some $d \geq 2$. Let $X$ be an indecomposable object in $\U$, and note that $X$ is perfect as $\Lambda$ has finite global dimension. By \cref{lem:all minimal}, the corresponding graded arc $\gamma_X$ is hence minimal. As the surface in the geometric model associated to $\Lambda$ only has one boundary component, \cref{lem:only minimal} part \textit{(2)} implies that $\U = \V_X^d$.

Since $X$ and $X_{n+1}$ are contained in $\V_X^d=\U$, \cref{lem:only minimal} part \textit{(1)} combined with $(\ast)$ yields that $d$ divides $n-1$. If $d < n-1$, there exists a non-minimal arc that crosses precisely $d+1$ edges of $L$. For this arc, we can choose a grading such that the corresponding object is contained in $\Ud=\U$. This contradicts \cref{lem:all minimal}, and we can thus conclude that $d=n-1$.

For the reverse direction, let $d=n-1$. As $n \geq 3$, this yields $d \geq 2$. Consider $\U = \V_X^d$ for some indecomposable object $X$ in $\D^b(\Lambda)$ corresponding to a minimal graded arc. Note that $\V_X^d$ is closed under $[d]$ by definition. It remains to show that $\V_X^d$ is a $d$-cluster tilting subcategory of$\D^b(\Lambda)$.

By the compatibility of grading in the definition of $\V_X^d$ combined with $(\ast)$ and the assumption $d=n-1$,  graded arcs corresponding to indecomposable objects in $\U = \V_X^d$ have $d$-compatible grading in common endpoints. This implies that
\[
\Hom(\V_X^d,\V_X^d[j]) = 0
\]
whenever $d \nmid j$. In particular, we have $\V_X^d \subseteq \Ud$ and $\V_X^d \subseteq \dU$.

Consider an indecomposable object $Y$ that is not contained in $\V_X^d$. If the corresponding graded arc $\gamma_Y$ is minimal, this means that the grading of $\gamma_Y$ in its endpoints is not $d$-compatible with the graded arcs corresponding to indecomposable objects in $\V_X^d$. This gives $Y \notin \Ud$ and $Y \notin\dU$. If $\gamma_Y$ is not minimal, we see that $\gamma_Y$ crosses precisely $l$ edges of $L$ for some $2 \leq l \leq d$. This implies that for any possible grading of $\gamma_Y$, one has $Y \notin \Ud$ and $Y \notin\dU$, which yields $\Ud \subseteq \V_X^d$ and $\dU \subseteq \V_X^d$. We can hence conclude that $\V_X^d = \Ud=\dU$.

It remains to observe that $\V_X^d$ is functorially finite in $\D^b(\Lambda$). For this, consider an indecomposable object $Y$ in the derived category. If the grading of $\gamma_Y$ in its endpoints is not compatible with the grading of any minimal arc corresponding to an object in $\V_X^d$, left and right approximations of $Y$ are given by the zero morphism. If the grading is compatible at one of the endpoints, we get left and right approximations by the morphisms corresponding to the graded intersection. Our subcategory $\U = \V_X^d$ is hence $d$-cluster tilting, which finishes the proof.
\end{proof}

Note that it is possible to show a version of \cref{prop:type A} by working directly with the AR-quiver of $\D^b(\Lambda)$ instead of using the geometric model.

\begin{remark}
As $\Lambda = KA_n/J^2$ is $d$-representation finite $d$-hereditary for $n=d+1$, the module category contains a unique $d$-cluster tilting subcategory $\U \subseteq \bmod \Lambda$. Notice that the $d$-cluster tilting subcategories described in the theorem above are equivalent to the subcategory $\U[d\Z]$ of $\D^b(\Lambda)$, see \cref{ex:d-CT} and \cref{ex:V_X^d}. Combining \cref{thm:main result} and \cref{prop:type A}, we see that all $d$-cluster tilting subcategories of the derived category of a gentle algebra that are closed under $[d]$ arise in this way. 
\end{remark}

Combining our results in this section with \cite{OPS}*{Corollary 1.23}, we obtain the following corollary.

\begin{corollary}
Let $\Lambda = KQ/I$ be a gentle algebra which is not a field. The following statements are equivalent:
\begin{enumerate}
    \item There exists a $d$-cluster tilting subcategory $\U \subseteq \D^b( \Lambda)$ that is closed under $[d]$ for some $d\geq 2$.
    \item The algebra $\Lambda$ is derived equivalent to an algebra of Dynkin type $A_n$ with $n \geq 3$.
    \item The quiver $Q$ is a tree with $\vert Q_0 \vert \geq 3$.
    \item The surface in the geometric model associated to $\Lambda$ is a disk with at least four marked points on the boundary.
\end{enumerate}
\end{corollary}

\begin{remark}
In the case where $\Lambda=K$ is a field, there exists a $d$-cluster tilting subcategory $\U_d \subseteq \D^b(\Lambda)$ that is closed under $[d]$ for any $d \geq 1$. This subcategory is given by
\[\U_d = 
\add\{K[di] \mid i \in \Z\} \subseteq \D^b(\Lambda),
\]
where the notation $K$ is used for the stalk complex with $K$ in degree $0$.
Any $d$-cluster tilting subcategory of $\D^b(\Lambda)$ is equivalent to $\U_d$.
\end{remark}

\begin{acknowledgements}
This work has been partially supported by project IDUN, funded through the Norwegian Research Council (295920).

The second author was partially funded by the Norwegian Research Council via the project "Higher homological algebra and tilting theory" (301046). The third author would like to thank the Isaac Newton Institute for Mathematical Sciences, Cambridge, for support and hospitality during the program Cluster Algebras and Representation Theory, where work on this paper was undertaken. This work was supported by EPSRC grant no EP/R014604/1.

Parts of this work was carried out while the first two authors participated in the Junior Trimester Program ``New Trends in Representation Theory'' at the Hausdorff Research Institute for Mathematics in Bonn. They thank the Institute for excellent working conditions. They would also like to thank Jenny August, Sondre Kvamme, Yann Palu and Hipolito Treffinger for helpful discussions.

We thank the anonymous referee for their helpful comments and suggestions.
\end{acknowledgements}


\end{document}